\documentclass[10pt,letterpaper]{article}

\usepackage{marvosym}
\usepackage[utf8]{inputenc}
\usepackage{amsmath}
\usepackage{amsfonts}
\usepackage{amssymb}
\usepackage{graphicx}
\usepackage{mathrsfs}
\usepackage{upref,amsthm,amsxtra,exscale}
\usepackage{cite}
\usepackage[colorlinks=true,urlcolor=blue,
citecolor=red,linkcolor=blue,linktocpage,pdfpagelabels,
bookmarksnumbered,bookmarksopen]{hyperref}
\usepackage{upgreek}

\usepackage{abstract}

\usepackage{amssymb,amsmath, amsfonts, amsthm,color}
\usepackage{graphicx}
\usepackage{mathrsfs}

\usepackage{latexsym}
\usepackage{color} 
\usepackage[dvips]{epsfig}
\usepackage{comment}
\usepackage{geometry}
\hypersetup{
    colorlinks,
    citecolor=blue,
    linkcolor=blue,
    urlcolor=black
}

\usepackage{bm}
\usepackage{dsfont}

\newtheorem{theorem}{Theorem}[section]
\newtheorem{lemma}[theorem]{Lemma}

\newtheorem{proposition}[theorem]{Proposition}

\newtheorem{remark}[theorem]{Remark}
\newtheorem{question}[theorem]{Question}
\newtheorem{ex-rem}[theorem]{Examples and Remarks}
\numberwithin{equation}{section}

\usepackage{color,graphicx,xcolor}

\newcommand{\eps}{\varepsilon}
\newcommand{\R}{{\mathbb R}}

\def\r{\mathbb{R}}

\def\rn{\mathbb{R}^n}

\def\eps{\varepsilon}

\def\io{\int_{\Omega}}
\def\irn{\int_{\r^n}}

\def\vr{\varrho}
\def\o{\Omega}
\def\t{\Theta}

\def\tilde{\widetilde}

\def\cC{\mathcal{C}}

\def\cN{\mathcal{N}}

\def\cU{\mathcal{U}}

\def\what{\widehat}

\def\d{\,\mathrm{d}}

\author{Mónica Clapp\footnote{ M. Clapp was supported by Progetto di Ricerca Ateneo Sapienza, ``Partial Differential Equations in Analysis and Geometry" P.I. Angela Pistoia and by UNAM-DGAPA-PAPIIT grant IA100923 (Mexico).
},
Angela Pistoia
\footnote{A. Pistoia was supported by
INDAM-GNAMPA project ``Problemi di doppia curvatura su variet\`a  a bordo e legami con le EDP di tipo ellittico'' and of the project 
``Pattern formation in nonlinear phenomena'' funded by the MUR Progetti di Ricerca di Rilevante Interesse Nazionale (PRIN) Bando 2022 grant 
20227HX33Z.}, and
Alberto Saldaña
\footnote{ A. Saldaña is supported by the 2024 Visiting Professor Program of La Sapienza University (Italy), CONAHCYT grant CBF2023-2024-116 (Mexico), and by UNAM-DGAPA-PAPIIT grants IA100923 and IN102925 (Mexico).
}
}
\title{Multiple solutions to a semilinear elliptic equation with a sharp change of sign in the nonlinearity}
\date{}

\begin{document}
\maketitle

\renewcommand{\abstractname}{Abstract}
\begin{abstract}
We consider a nonautonomous semilinear elliptic problem where the power-type nonlinearity is multiplied by a discontinuous coefficient that takes the value one inside a bounded open set $\Omega$ and minus one in its complement. In the slightly subcritical regime, we prove the existence of concentrating  positive and nodal solutions.  Moreover, depending on the geometry of $\Omega$, we establish multiplicity of positive solutions. Finally, in the critical case, we show the existence of a blow-up positive solution when $\Omega$ has nontrivial topology.  Our proofs rely on a Lyapunov-Schmidt reduction strategy which in these problems turns out to be remarkably simple. We take this opportunity to highlight certain aspects of the method that are often overlooked and present it in a more accessible and detailed manner for nonexperts.

\medskip


\smallskip

\end{abstract}

\section{Introduction} 
\label{sec:intro}
In this paper we construct multiple solutions to the following nonautonomous semilinear elliptic problem
\begin{equation}\label{eq:problem}
\begin{cases}
-\Delta u= Q_\Omega(x) |u|^{q-1}u,\\
u\in E:=D^{1,2}(\mathbb R^n)\cap L^{q+1}(\rn),
\end{cases}
\end{equation}
where $1<q\leq p:=\frac{n+2}{n-2}=2^*-1$, $\o\subset \rn$ is a bounded (not necessarily connected) smooth open subset, $n\geq 3$, and \  $Q_\Omega:={\mathds 1}_\Omega-{\mathds 1}_{\mathbb R^n\setminus\Omega}$, \ i.e.,
\begin{equation*}
Q_\Omega(x)=
\begin{cases}
1 & \text{if \ }x\in\Omega, \\
-1& \text{if \ }x\in\mathbb R^n\setminus\Omega.
\end{cases}
\end{equation*}
Here, as usual, $2^*:=\frac{2n}{n-2}$ denotes the critical Sobolev exponent and $D^{1,2}(\rn):=\{u\in L^{2^*}(\rn):\nabla u\in L^2(\rn,\rn)\}$.

This problem arises in connection with some models of optical waveguides propagating through a stratified dielectric medium, such as those in \cite{st1,st2}. As the self-focusing part of the dielectric response shrinks, the amplitude of the electric field concentrates and blows up, see \cite{as}. It is shown in \cite{fw} that its limiting profile is a least energy solution of \eqref{eq:problem}.

Problem \eqref{eq:problem} also arises from a system of competing species, each of which is attracted to a different region in space and is repelled from its complement. As the region of attraction shrinks, the components of the system concentrate and blow up and their limiting profile is, again, a solution of \eqref{eq:problem}, see \cite{css}. 

It is therefore interesting to study problem \eqref{eq:problem} and to describe the shape and qualitative properties of its solutions. Some results on their symmetry properties and their decay rate at infinity were recently established in \cite{CHS24}.

In this work, we use a Lyapunov-Schmidt reduction method to study \eqref{eq:problem} in three settings:
\begin{enumerate}
 \item Existence and multiplicity of positive solutions in the slightly subcritical regime, namely, when $q=\frac{n+2}{n-2}-\eps$ and $\eps>0$ is a small parameter.
 \item Existence of nodal solutions also in the slightly subcritical regime.
 \item Existence of a positive solution for the critical exponent $q=\frac{n+2}{n-2}$ in domains with a shrinking hole.
\end{enumerate}

The solutions that we construct exhibit concentration as some parameter goes to zero and have \emph{bubbles} as profiles. To be more precise, recall that the positive solutions to the critical problem
\begin{equation}\label{eq:yamabe}
-\Delta u=u^{p },\qquad u\in D^{1,2}(\rn),
\end{equation}
are the {\em standard bubble}
\begin{align*}
U(x):=\frac{\alpha_n}{(1+|x|^2)^\frac{n-2}{2}},\qquad \alpha_n:= (n(n-2))^\frac{n-2}{4}, \ x\in \rn,
\end{align*}
and its translations and dilations
\begin{equation}\label{bub}
U_{\delta,\xi}(x)
:=\delta^{-{\frac{n-2}{2}}}U\Big(\frac{x-\xi}{\delta}\Big)
=\frac{\alpha_n \delta^{{\frac{n-2}{2}}}}{(\delta^2+|x-\xi|^2)^\frac{n-2}{2}},
\qquad \xi\in\rn, \ \delta>0,\ x\in \rn.
\end{equation}

Our first result shows that the slightly subcritical problem has a positive solution for any bounded open set $\Omega\subset \rn$ and characterizes its concentration behavior as $\eps\to 0$. We write $\|\cdot\|$ for the norm in $D^{1,2}(\rn)$ and $|\cdot|_q$ for the norm in $L^q(\rn)$.

 \begin{theorem}
 \label{main1}
There exists $\eps_0>0$ such that, for each $\eps\in(0,\eps_0)$, the problem
\begin{equation}\label{eq:p}
\begin{cases}
-\Delta u= Q_\o |u|^{p-1-\eps}u,\\
u\in D^{1,2}(\mathbb R^n)\cap L^{p+1-\eps}(\rn),
\end{cases}
\end{equation}
has a positive solution of the form
\begin{align}\label{ueps}
 u_\eps=U_{\delta(\eps),\xi(\eps)}+\phi_\eps,\quad \text{ where \ }\max\{\|\phi_\eps\|,|\phi_\eps|_{p+1-\eps}\}\to 0\text{ as }\eps\to 0.
\end{align}
Here, the blow-up rate and the concentration point satisfy that
\begin{align}\label{ueps2}
\delta(\eps)\eps^{-\frac{1}{n}}\to d \quad \text{ and }\quad
\xi(\eps)\to \xi_0\quad \text{as $\eps\to 0$,}
\end{align}
where $d>0$ and $\xi_0\in\Omega$ is a global minimum of the function
\begin{equation}\label{psio}
 \psi_\Omega:\Omega\to\R\qquad \text{ given by }\qquad
\psi_\Omega(\xi):=\int_{\mathbb R^n\setminus \Omega} \frac1{|x-\xi|^{2n}}\d x.\end{equation}
\end{theorem}
Note that the function $ \psi_\Omega$ is singular at $\partial \Omega$ and smooth in the interior of $\Omega$ (see Figure \ref{fig1} below).  Hence, the existence of a global minimum is always guaranteed.

Existence of positive (ground state) solutions to \eqref{eq:problem} for any $q\in (1,p)$ is already known, see \cite{fw} for the case when $\o$ is the unit ball and \cite{CHS24} for general $\o$.  Theorem~\ref{main1} gives a detailed characterization of their blow-up behavior as the exponent approaches the critical one.

Our method also allows us to prove the following multiplicity result.

\begin{theorem}
\label{main11}
Let $\xi_0\in\Omega$ be a non-degenerate critical point of the function $\psi_\Omega$ given by \eqref{psio}.  Then, there exists $\eps_0>0$ such that, for each $\eps\in(0,\eps_0),$ the problem \eqref{eq:p} has a positive solution $u_\eps$ satisfying \eqref{ueps} and \eqref{ueps2}.
 \end{theorem}

The assumption that the critical point $\xi_0$ must be non-degenerate is not too restrictive. We show in the appendix (see Theorem~\ref{mainap}) that for almost every small $\cC^2$-perturbation $\o_\theta$ of a given bounded open set $\Omega\subset \rn$, the function $\psi_{\Omega_\theta}$ is a \emph{Morse function}, that is, all of its critical points are nondegenerate. In other words, the determinant of the Hessian at every critical point of $\psi_{\Omega_\theta}$ is not zero.

Theorem~\ref{main11} implies that, in some domains, \eqref{eq:p} has multiple positive solutions.

\begin{ex-rem}
\emph{
\begin{itemize}
\item[$(a)$] If $\o=B(\xi,r)$ is the ball of radius $r$ and center $\xi$, then $\xi$ is a critical point of $\psi_\o$ and it is a global minimum.
\item[$(b)$] If $\o=\o_1\cup\cdots\cup\o_m$ with $\overline{\o}_i\cap\overline{\o}_j=\emptyset$ if $i\neq j$, then, since $\lim_{\xi\to\partial\o}\psi_\o(\xi)=\infty$, the function $\psi_\o$ has $m$ local minima $\xi_i\in\o_i$, $i=1,\ldots,m$.
\item[$(c)$] If $\o=B(\xi_1,r_1)\cup B(\xi_2,r_2)$, the balls are disjoint and $r_1>r_2$, then the global minimum of $\psi_\o$ is close to $\xi_1$.
\item[$(d)$] Let $\Omega$ be as in $(c)$ and let $L$ be a thin tube connecting the two balls.  Then $D:=\Omega \cup L$ is a dumbbell domain. Let $x_1$ be the global minimum of $\psi_\o$ (which is close to $\xi_1$) and let $x_2$ be the local minimum close to $\xi_2$. Then, there is $r\in(0,r_2)$ such that
\begin{align*}
 \inf_{x\in \partial B_r(x_2)}\psi_D(x) >\psi_D(x_2)\geq \psi_D(x_1).
\end{align*}
Hence, the mountain pass theorem applied to $\psi_D$ yields the existence of a critical saddle point $x^*\in L.$ In Figure~\ref{fig1} we include a picture of the function $\psi_D$ for $n=2$ (our results only consider $n\geq 3,$ but the shape of $\psi_D$ is analogous).  Notice the existence of a global minimum (in the large ball), a local minimum (in the small ball), and a saddle point (on the connecting bridge).
\begin{figure}[h!]
\centering
 \includegraphics[width=9cm]{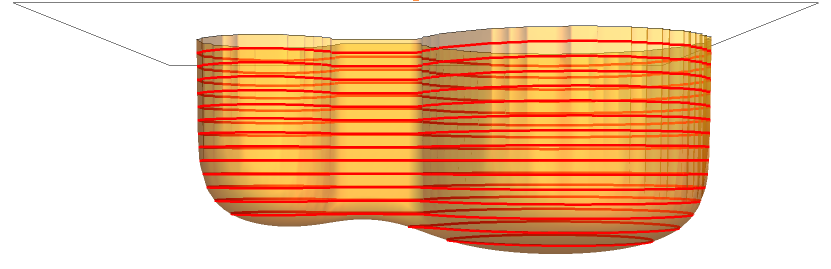}
 \includegraphics[width=7cm]{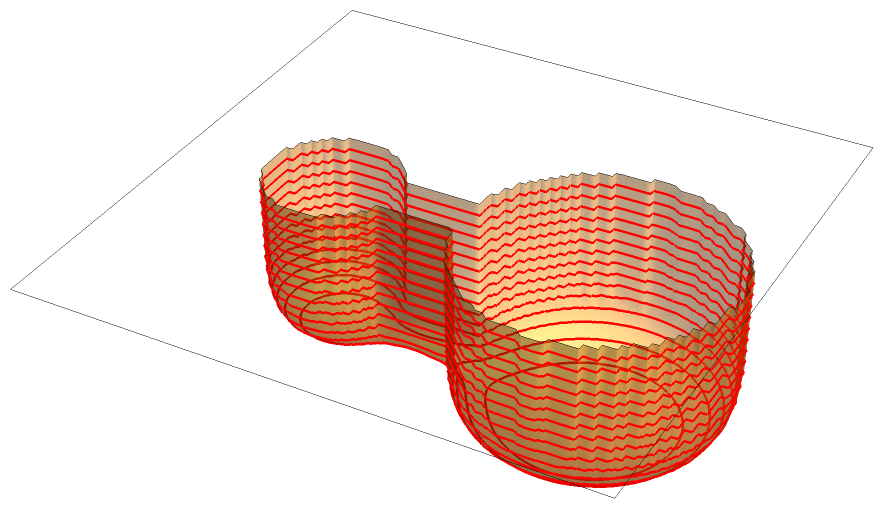}
 \caption{A (truncated) plot of the function $\psi_D$ for $n=2$. }\label{fig1}
\end{figure}
\item[$(e)$] $\psi_\o$ has at least $\mathrm{cat}(\o)$ critical points, where $\mathrm{cat}$ stands for the Lusternik-Schnirelmann  category. As mentioned above, they are generically nondegenerate with respect to $\o$. So, generically, problem \eqref{eq:p} has at least $\mathrm{cat}(\o)$ positive solutions.
\end{itemize}
}
\end{ex-rem}

Next we turn our attention to the existence of \emph{nodal solutions}. In this case, we use as blow-up profiles the difference of two bubbles concentrating near the boundary of $\o.$
 
\begin{theorem} \label{main2}
Let $\o$ be a smooth bounded open set. There exists $\eps_0>0$ such that, for each $\eps\in(0,\eps_0)$, the problem \eqref{eq:p} has a sign-changing solution $u_\eps$ of the form
\begin{align}\label{ueps:n}
 u_\eps = U_{\delta_1(\eps),\xi_1(\eps)}-U_{\delta_2(\eps),\xi_2(\eps)}+\phi_\eps,\quad \text{ where }\max\{\|\phi_\eps\|,|\phi_\eps|_{p+1-\eps}\}\to 0\text{ as }\eps\to 0.
\end{align}
Here, the blow-up rates and the concentration points $\xi_1(\eps),\xi_2(\eps)\in \o$ satisfy that
\begin{align}\label{ueps2:n}
\delta_i(\eps)\eps^{-\frac{1}{n-2}}\to d_i>0 \quad \text{ and }\quad
\xi_i(\eps)\to \eta_i\quad \text{as $\eps\to 0$, $i=1,2$,}
\end{align}
where $\eta_1,\eta_2\in\partial\Omega$ are such that
\begin{align*}
|\eta_1-\eta_2|=\max\{|x_1-x_2|: x_1,x_2\in\partial\o\}.
\end{align*}
\end{theorem} 

Least energy sign-changing solutions for \eqref{eq:problem} for any $q\in(1,p)$ were shown to exist in \cite{CHS24}. We do not know if the nodal solutions in Theorem~\ref{main2} have least energy or not, but their energy is close to the least possible energy for nodal solutions.

Note that the concentration points in Theorem~\ref{main2} are ``almost as far as possible'' from each other, while still being slightly far away from the ``repellent zone'' $\rn\backslash\o$, where the coefficient $Q$ is negative.  A much more detailed information on the behavior of the concentration points is available in Section~\ref{sch:sec}.

\medskip

For our last result, we assume that $0\in\o$ and we consider the \emph{critical problem}
\begin{equation}\label{eq:p_critical}
\begin{cases}
-\Delta u= Q_{\o_{\vr}} |u|^{p-1}u,\\
u\in D^{1,2}(\mathbb R^n),
\end{cases}
\end{equation}
where $\o_\vr:=\{x\in\o:|x|>\vr\}$, $\vr>0$ is a small parameter, and $Q_{\o_\vr}:={\mathds 1}_{\o_\vr}-{\mathds 1}_{\mathbb R^n\setminus{\o_\vr}}$.  This means that the set $\o$ has a small hole around the origin (and therefore, $\o$ has a nontrivial topology). In this setting, we show the existence of a positive solution to \eqref{eq:p_critical} concentrating close to the center of the small hole, whenever its radius is sufficiently small.

\begin{theorem}\label{main3}
There exists $\vr_0>0$ such that, for each $\vr\in(0,\vr_0)$, the problem \eqref{eq:p_critical} has a positive solution $u_\vr$ of the form
\begin{align*}
 u_\eps=U_{\delta(\vr),\xi(\vr)}+\phi_\vr,\quad \text{ where }\|\phi_\vr\|\to 0\text{ as }\vr\to 0.
\end{align*}
Here, the blow-up rate and the concentration point satisfy that
\begin{align*}
\delta(\vr)\vr^{-\frac{1}{2}}\to d>0 \quad \text{ and }\quad
\xi(\vr)\to 0\quad \text{as $\vr\to 0$,}
\end{align*}
\end{theorem}

We stress that the critical problem \eqref{eq:problem} with $q=\frac{n+2}{n-2}$ does not have a least energy solution. Furthermore, it does not have a nontrivial solution when $\o$ is strictly starshaped. These and other results may be found in \cite{cfs}. So the solutions given by Theorem \ref{main3} are not least energy solutions, but their energy approaches the least possible energy as the hole shrinks.

The proof of our theorems follows the Lyapunov-Schmidt reduction procedure, which has been successfully applied to a wide variety of elliptic problems, see for instance \cite{ackclpi,blp,gmp,bamipi,mupi,mp,mupi1,mupi2,PST23,cpps} for some of the papers that inspired our proofs and where similar estimates can be found for other problems.

In the elliptic setting, this technique is well known and it has proved to be a powerful tool to construct solutions to critical or almost critical problems. But the method is also known for its technical difficulty and its need for long and sharp calculations.  Due to this high degree of complexity, the experts on the method often omit some technical details that are known to be true. This has the advantage of making the arguments shorter and more readable, but, on the other hand, it also makes the machinery more obscure for nonexperts.  Luckily, when applying this technique to \eqref{eq:problem}, one has some advantages: first, since \eqref{eq:problem} has no boundary conditions, the use of Green or Robin functions\textemdash which are often involved in delicate expansions\textemdash is not needed. Moreover, the differential operator (the Laplacian) and the nonlinearity (of power type) are rather simple and this eases the calculations.  Hence, one of the contributions that we aim for in this paper, is to give our own perspective on some key ingredients in the method that are often left aside, such as the pivotal role of rescalings as changes of coordinates between the linearizations at the bubbles and the detailed application of the implicit function theorem and the Banach fixed point theorem, which are crucial. We pay special attention to the proofs of Theorems~\ref{main1} and~\ref{main11} (see Proposition \ref{rid1}), where we tried to be as clear and simple as possible and to show, in a transparent way, the intuition behind the calculations.  Then the proofs of Theorems~\ref{main2} and~\ref{main3} are shorter by referring to their counterparts in the previous sections.

We hope that this presentation can be of help to researchers that are not very familiar with the Lyapunov-Schmidt reduction method and wish to understand it better.  

We close this introduction with an open problem: 

\begin{question}
If $\o$ is a convex set, is it true that $\psi_\o$ has a unique critical point?
\end{question}
We believe that this is indeed true.

\medskip

The paper is organized as follows.  In Section~\ref{sec:preliminaries} we present some preliminaries, including some standard notation and tools to develop the method. Theorems~\ref{main1} and~\ref{main11} are proved in Section~\ref{sec:positive solution}, where positive solutions for the subcritical problem are studied. Then, nodal solutions are investigated in Section~\ref{sch:sec}, where Theorem~\ref{main2} is proved. Section~\ref{sec:critical} is devoted to the critical problem \eqref{eq:p_critical} and to the proof of Theorem~\ref{main3}.  Finally, in the appendix, we include some auxiliary computations and we show that the function $\psi_\Omega$ given in \eqref{psio} is generically a Morse function.

\section{Preliminaries}
\label{sec:preliminaries}

\subsection{A decomposition of the space}

We begin by fixing some standard notation. The space
\begin{align}\label{Ddef}
D^{1,2}(\rn):=\{u\in L^{2^*}(\rn):\nabla u\in L^2(\rn,\rn)\}
\end{align}
is a Hilbert space with its usual inner product and norm
 \begin{align*}
 \langle u,v\rangle:=\irn\nabla u\cdot\nabla v,\qquad \|u\|:=\left(\irn |\nabla u|^2\right)^{1/2}.
 \end{align*}
We write
 \begin{align*}
  |u|_q:=\left( \int_{\rn}|u|^q \right)^\frac{1}{q}\qquad \text{ for \ }u\in L^q(\rn),\qquad q\geq 1.
 \end{align*}
 
Fix $\frac{n}{n-2}<s\leq 2^*$ and define 
$$E:=D^{1,2}(\rn)\cap L^{s}(\rn),$$ 
endowed with the norm $\|u\|_E:=\max\{\|u\|,|u|_{s}\}.$ Set $p:=2^*-1=\frac{n+2}{n-2}$. Recall that the positive solutions to \eqref{eq:yamabe} are the bubbles $U_{\delta,\xi}$ given in \eqref{bub}.  Note that, as $s>\frac{n}{n-2}$, the functions $U_{\delta,\xi}$ belong to $E$. Moreover,
$$\partial_{\delta} U_{\delta,\xi}(x)=\alpha_n\frac{n-2}{2}\delta^\frac{n-4}{2}\frac{|x-\xi|^2-\delta^2}{(\delta^2+|x-\xi|^2)^{n/2}},$$
and
$$\partial_{\xi_i} U_{\delta,\xi}(x)=\alpha_n(n-2)\delta^\frac{n-2}{2}\frac{x_i-\xi_i}{(\delta^2+|x-\xi|^2)^{n/2}},\qquad i=1,\ldots,n,$$
generate the space of solutions\textemdash denoted $K_{\delta,\xi}$\textemdash to the linear problem
\begin{equation}\label{eq:linear equation}
-\Delta z= p U_{\delta,\xi}^{p-1}z,\qquad z\in D^{1,2}(\mathbb R^n).
\end{equation}
For each $\delta>0$ and $\xi\in\rn,$ the rescaling function $P_{\delta,\xi}:D^{1,2}(\rn)\to D^{1,2}(\rn)$ given by
\begin{equation} \label{eq:rescaling}
P_{\delta,\xi}u(x):=\delta^{-\frac{n-2}{2}}u\Big(\frac{x-\xi}{\delta}\Big)
\end{equation}
is a linear isometry, i.e.,
$$\|P_{\delta,\xi}u\|=\|u\|\qquad\text{for all \ } u\in D^{1,2}(\rn),$$
and satisfies
\begin{equation} \label{eq:s_rescaling}
|P_{\delta,\xi}u|_s= \delta^\alpha |u|_s\qquad\text{for all \ } u\in E,\qquad\text{where \ }\alpha:=\frac{n}{s}-\frac{n-2}{2}.
\end{equation}
Note that $\alpha>0$ if $s<2^*$. Set
$$Z_{1,0}^0(x):=\alpha_n\frac{n-2}{2}\frac{|x|^2-1}{(1+|x|^2)^{n/2}}\qquad\text{and}\qquad Z_{1,0}^i(x):=\alpha_n(n-2)\frac{x_i}{(1+|x|^2)^{n/2}},\qquad i=1,\ldots,n,$$
and define
\begin{align}\label{z:def}
Z_{\delta,\xi}^i:=P_{\delta,\xi}Z_{1,0}^i,\qquad i=0,1,\ldots,n.
\end{align}
Observe that $Z_{\delta,\xi}^0=\delta\partial_\delta U_{\delta,\xi}$ and $Z_{\delta,\xi}^i=\delta\partial_{\xi_i} U_{\delta,\xi}$ for $i=1,\ldots,n$. Therefore,
\begin{equation}\label{k}
K_{\delta,\xi}=\mathrm{span}\{Z_{\delta,\xi}^{i}: i=0,1,\dots,n\}=P_{\delta,\xi}K_{1,0}.
\end{equation}
Consider its orthogonal complement in $D^{1,2}(\rn)$,
\begin{equation}\label{ko}
K_{\delta,\xi}^\perp:=\left\{u\in D^{1,2}(\mathbb R^n):\langle Z_{\delta,\xi}^{i},u\rangle=0,\ i=0,1,\dots,n\right\},
\end{equation}
and the orthogonal projections in $D^{1,2}(\rn)$,
\begin{align*}
\Pi_{\delta,\xi}:D^{1,2}(\mathbb R^n)\to K_{\delta,\xi}\quad\text{ and }\quad
\Pi_{\delta,\xi}^\perp:D^{1,2}(\mathbb R^n)\to K_{\delta,\xi}^\perp.
\end{align*}
Since $P_{\delta,\xi}$ is an isometry of $D^{1,2}(\rn)$ we have that 
\begin{equation} \label{eq:proyecciones}
K_{\delta,\xi}^\perp=P_{\delta,\xi}K_{1,0}^\perp,\qquad \Pi_{\delta,\xi}=P_{\delta,\xi}\Pi_{1,0}P^{-1}_{\delta,\xi},\qquad\text{and}\qquad\Pi_{\delta,\xi}^\perp=P_{\delta,\xi}\Pi_{1,0}^\perp P^{-1}_{\delta,\xi}.
\end{equation}

The following lemma establishes the (uniform) continuity of the projections.
\begin{lemma} \label{lem:projections}
\hspace{2em}
\begin{itemize}
\item[$(i)$]$K_{\delta,\xi}\subset E$.
\item[$(ii)$]If $u\in E$, then $\Pi_{\delta,\xi}^\perp u\in K_{\delta,\xi}^\perp\cap L^s(\rn)\subset E$.
\item[$(iii)$]The restriction to $E$ of the orthogonal projections,
\begin{align*}
\Pi_{\delta,\xi}:E\to K_{\delta,\xi}\quad\text{ and }\quad
\Pi_{\delta,\xi}^\perp:E\to K_{\delta,\xi}^\perp\cap L^s(\rn),
\end{align*}
are continuous as functions from $E$ to $E$. Furthermore, there exist positive constants $C_1,C_2$ such that
$$\|\Pi_{\delta,\xi}u\|_E\leq C_1 \|u\|_E\quad\text{ and }\quad \|\Pi_{\delta,\xi}^\perp u\|_E\leq C_2 \|u\|_E\qquad\text{ for all }u\in E, \ \delta\in(0,1),\text{ and }\xi\in\rn.$$
\end{itemize}
\end{lemma}

\begin{proof}
$(i):$ \ Since $s>n/(n-2)$, a direct calculation shows that $Z_{\delta,\xi}^{i}\in L^s(\rn)$ for all $i=0,1,\ldots,n$.

$(ii):$ \ As $\Pi_{\delta,\xi}^\perp u=u-\Pi_{\delta,\xi} u$, this statement follows from $(i)$.

$(iii):$ \ Since $K_{1,0}$ is finite dimensional there exists $C_1\geq 1$ such that $|v|_s\leq C_1\|v\|$ for every $v\in K_{1,0}$. Then, by \eqref{eq:proyecciones} and
\eqref{eq:s_rescaling},
\begin{align*}
 |\Pi_{\delta,\xi}u|_s
 =\delta^\alpha |\Pi_{1,0}P_{\delta,\xi}^{-1}u|_s
 \leq   C_1 \|\Pi_{1,0}P_{\delta,\xi}^{-1}u\|\leq C_1\|u\|\leq C_1\|u\|_E\qquad\text{for all \ }u\in E, \ \delta\in(0,1),\text{ and }\xi\in\rn.
\end{align*}
Clearly, $\|\Pi_{\delta,\xi}u\|\leq \|u\|\leq \|u\|_E$ for every $u\in E$. Therefore, there is $C_2>0$ such that
$$\|\Pi_{\delta,\xi}u\|_E\leq C_2\|u\|_E\qquad\text{for all \ }u\in E, \ \delta\in(0,1),\text{ and }\xi\in\rn.$$
Finally, as $\Pi_{\delta,\xi}^\perp u=u-\Pi_{\delta,\xi} u$, there exists $C_3$ such that $\|\Pi_{\delta,\xi}^\perp u\|_E\leq C_3 \|u\|_E$ for every $u\in E$, $\delta\in(0,1)$, and $\xi\in\rn$, as claimed.
\end{proof}

\subsection{The adjoint problem}

Let
\begin{align*}
i^*:L^{\frac{2n}{n+2}}(\mathbb R^n)\to D^{1,2}(\mathbb R^n)
\end{align*}
be the adjoint operator of the embedding $i:D^{1,2}(\mathbb R^n)\hookrightarrow L^\frac{2n}{n-2}(\mathbb R^n)$, i.e., if $g\in L^\frac{2n}{n+2}(\rn)$, then $u=i^*g$ if and only if $u$ is a solution of 
\begin{align}\label{ug}
-\Delta u=g,\qquad u\in D^{1,2}(\rn).
\end{align}
Note that $i^*$ can also be seen as the solution mapping of \eqref{ug}, namely, the convolution of $g$ with the fundamental solution.

Let $q\in[\frac{ns+2s}{2n},\frac{n+2}{n-2}]$ and consider the problem
\begin{equation}\label{eq:problem+}
\begin{cases}
-\Delta u= Q_\o |u|^{q-1}u,\\  
u\in E.
\end{cases}
\end{equation}
As $s\leq q\frac{2n}{n+2}\leq 2^*$, if $u\in E$, then, by interpolation, $u^q\in L^\frac{2n}{n+2}(\rn)$ and $i^*(Q_\o |u|^{q-1}u)$ is well defined. Thus, $u$ solves \eqref{eq:problem+} if and only if $u$ solves
\begin{equation} \label{eq:adjoint_problem}
\begin{cases}
u= i^*(Q_\o |u|^{q-1}u),\\  
u\in E.
\end{cases}
\end{equation}

\begin{lemma}\hspace{2em} \label{lem:holder}
\begin{itemize}
\item[$(i)$]If $g\in L^\frac{2n}{n+2}(\rn)\cap L^\frac{sn}{n+2s}(\rn)$, then $i^*g\in E$ and there is a constant $C=C(n,s)>0$ such that
$$\|i^*g\|_E\leq C\max\{|g|_\frac{2n}{n+2},|g|_\frac{sn}{n+2s}\}.$$
\item[$(ii)$]If $v\in L^\frac{n}{2}(\rn)$ and $u\in E$, then $i^*(vu)\in E$ and there is a constant $C=C(n,s)>0$ such that
$$\|i^*(vu)\|_E\leq C|v|_\frac{n}{2}\|u\|_E.$$
\end{itemize}
\end{lemma}

\begin{proof}
$(i):$ \ Since $i^*$ is continuous, there is a constant $C=C(n)>0$ such that $\|i^*g\|\leq C|g|_{\frac{2n}{n+2}}$. Furthermore, from the Hardy-Littlewood-Sobolev inequality we get that $i^*g\in L^s(\rn)$ and there is a constant $C=C(n,s)>0$ such that $|i^*g|_s\leq C|g|_{\frac{ns}{n+2s}}$,
see \cite[Remark 2.9]{mp}. The statement follows.

$(ii):$ \ Hölder's inequality yields $vu\in L^\frac{2n}{n+2}(\rn)\cap L^\frac{sn}{n+2s}(\rn)$ and, using also Sobolev's inequality,
$$|vu|_\frac{2n}{n+2}\leq |v|_\frac{n}{2}|u|_{2^*}\leq C|v|_\frac{n}{2}\|u\|_E\qquad\text{and}\qquad |vu|_\frac{sn}{n+2s}\leq |v|_\frac{n}{2}|u|_s\leq |v|_\frac{n}{2}\|u\|_E.$$
Now the statement follows from $(i)$.
\end{proof}

\subsection{Some useful inequalities}

For future reference, we quote the following well known inequalities, see \cite[Lemma 2.2]{l}.

\begin{lemma}\label{lem:li}
Given $r\geq 1$ there is a positive constant $C=C(r)$ such that for all $a,b\in\r$ with $a\geq0$,
\begin{equation*}
||a+b|^{r}-a^{r}|\leq
\begin{cases}
C\min\{|b|^{r},a^{r-1}|b|\} &\text{if \ }0<r<1,\\
C(|b|^{r}+a^{r-1}|b|) & \text{if \ }r\geq1,
\end{cases}
\end{equation*}
and
\begin{equation*}
||a+b|^{r}(a+b)-a^{r+1}-(r+1)a^rb|\leq
\begin{cases}
C\min\{|b|^{r+1},a^{r-1}b^2\} &\text{if \ }0\leq r<1,\\
C\max\{|b|^{r+1},a^{r-1}b^2\} & \text{if \ }r\geq1.
\end{cases}
\end{equation*}
\end{lemma}

\section{The slightly subcritical problem: existence of a positive solution}
\label{sec:positive solution}

From now on, we fix
\begin{align*}
s\in \left(\frac{2n}{n+2}\,,\,2^*\right).
\end{align*}
To prove Theorem~\ref{main1} we show that, for $\eps>0$ small enough, the problem
\begin{equation}\label{p2}
u-i^* \left[Q_\Omega  f_\eps (u)\right]=0,\qquad   u\in E,\qquad f_\eps(u):=|u|^{p-1-\eps}u,
\end{equation}
has a solution of the form $u:=U_{\delta,\xi}+\phi$, with
\begin{equation}\label{para}
\delta:=d\eps^\frac{1}{n}\ \hbox{with}\ d\in(0,\infty), \qquad \xi\in\Omega,\qquad\hbox{and}\qquad \phi\in K_{\delta,\xi}^\perp\cap L^s(\rn).
 \end{equation}
The  choice of $\delta$ is motivated by the leading terms in the expansion of the energy given in the proof of the Proposition~\ref{red-en}, see Remark \ref{rate:rem}.

Note that $U_{\delta,\xi}+\phi$ solves \eqref{p2} if and only if it solves the system
\begin{align}\label{s}
\Pi_{\delta,\xi}^\perp\left[U_{\delta,\xi}+\phi-i^* \left[Q_\Omega  f_\eps \left(U_{\delta,\xi}+\phi\right)\right]\right]=0,\qquad
\Pi_{\delta,\xi} \left[U_{\delta,\xi}-i^* \left[Q_\Omega  f_\eps \left(U_{\delta,\xi}+\phi\right)\right]\right]=0.
\end{align}
Our aim is to establish the existence of an error $\phi$ and parameters $\delta$ and $\xi$ satisfying these two equations.

\subsection{Existence of the error}

First, we focus our attention on solving the first equation in \eqref{s}. We will show that, if $\eps$ is small enough, this equation has a unique solution $\phi_{\delta,\xi}\in K_{\delta,\xi}^\perp\cap L^s(\rn)$ for each pair of parameters $(\delta,\xi)$, and we estimate the norm of the solution. We start with the following lemma.

\begin{lemma} \label{lem:estimates}
Given a compact subset $Q$ of $(0,1)\times\o$, there exist $\eps_1>0$ and $C>0$ such that, for every $\eps\in(0,\eps_1)$, $(d,\xi)\in Q$, and $\delta=d\eps^\frac{1}{n}$, the following statements hold true:
\begin{itemize}
\item[$(i)$]The function $\mathscr N:K^\perp_{\delta,\xi}\cap L^s(\rn)\to K^\perp_{\delta,\xi}\cap L^s(\rn)$, given by
$$\mathscr N(\phi):=(\Pi_{\delta,\xi}^\perp\circ i^*)\left[ Q_\Omega\left[  f_\eps \left(U_{\delta,\xi}+\phi \right) -
 f_\eps \left(U_{\delta,\xi}  \right) - f'_\eps \left(U_{\delta,\xi}\right)\phi  \right]\right],$$
is well defined and
\begin{align}\label{N:ref}
\|\mathscr N(\phi)\|_E\leq C(\|\phi\|_E^2 + \|\phi\|_E^{p-\eps})\qquad\text{for every \ }\phi\in K^\perp_{\delta,\xi}\cap L^s(\rn).
\end{align}
\item[$(ii)$]The function $\mathscr L_\eps:K^\perp_{\delta,\xi}\cap L^s(\rn)\to K^\perp_{\delta,\xi}\cap L^s(\rn)$, given by
$$\mathscr L_\eps (\phi):=(\Pi_{\delta,\xi}^\perp\circ i^*)\left[ Q_\Omega\left[  f'_\eps \left(U_{\delta,\xi}  \right)
 - f'_0\left(U_{\delta,\xi}\right)  \right]\phi\right]
 +2(\Pi_{\delta,\xi}^\perp\circ i^*)\left[-{\mathds 1}_{\mathbb R^n\setminus\Omega}  f'_0\left(U_{\delta,\xi}\right)\phi\right],$$
is well defined and
\begin{align}\label{re3}
\|\mathscr L_\eps(\phi)\|_E\leq C(\eps|\ln \eps|+\delta^2)\|\phi\|_E \qquad\text{for every \ }\phi\in K^\perp_{\delta,\xi}\cap L^s(\rn).
\end{align}
\item[$(iii)$]The quantity
\begin{align*}
\mathscr E:=(\Pi_{\delta,\xi}^\perp\circ i^*)\left[Q_\Omega\left[f_\eps\left(U_{\delta,\xi}\right)
 - f_0\left(U_{\delta,\xi}\right)  \right]\right]
 +2(\Pi_{\delta,\xi}^\perp\circ i^*)\left[
 -{\mathds 1}_{\mathbb R^n\setminus\Omega}  f_0\left(U_{\delta,\xi}\right)\right]
\end{align*}
is well defined and
$\|\mathscr{E}\|_E\leq C(\eps|\ln \eps|+\delta^\frac{n+2}{2}).$
\end{itemize}
\end{lemma}

\begin{proof}
$(i):$ \ By the mean value theorem, for each $x\in\rn$, there exists $t\in(0,1)$ such that
$$|f_\eps(U_{\delta,\xi}+\phi) - f_\eps(U_{\delta,\xi}) - f'_\eps (U_{\delta,\xi})\phi|=|f'_\eps(U_{\delta,\xi}+t\phi)-f'_\eps(U_{\delta,\xi})||\phi|,$$
and, by Lemma~\ref{lem:li},
\begin{equation*}
|f'_\eps(U_{\delta,\xi}+t\phi)-f'_\eps(U_{\delta,\xi})|\leq C
\begin{cases}
|\phi|^{p-\eps-1} &\text{if \ }1<p-\eps<2,\\
U_{\delta,\xi}^{p-\eps-2}|\phi|+|\phi|^{p-\eps-1} &\text{if \ }2\leq p-\eps.
\end{cases}
\end{equation*}
Then, Hölder's inequality implies that $|f'_\eps(U_{\delta,\xi}+t\phi)-f'_\eps(U_{\delta,\xi})|\in L^\frac{n}{2}(\rn)$ and
$$|f'_\eps(U_{\delta,\xi}+t\phi)-f'_\eps(U_{\delta,\xi})|_\frac{n}{2}\leq C(|\phi|_{\frac{n}{2}(p-\eps-1)}+|\phi|_{\frac{n}{2}(p-\eps-1)}^{p-\eps-1})\leq C(\|\phi\|_E + \|\phi\|_E^{p-\eps-1}),$$
if $s\leq \frac{n}{2}(p-\eps-1)\leq 2^*$. Now the statement follows from Lemmas~\ref{lem:holder} and~\ref{lem:projections}.

$(ii):$ \ Using the mean value theorem as in \cite[Remark 2.21]{mp} one shows that $f'_\eps(U_{1,0}) - f'_0(U_{1,0})\in L^\frac{n}{2}(\rn)$ for $\eps$ small enough, and that
\begin{align}\label{re1}
|f'_\eps(U_{1,0}) - f'_0(U_{1,0})|_\frac{n}{2}\leq C\eps.
\end{align}
Since $U_{\delta,\xi}=P_{\delta,\xi}U_{1,0}$, performing a change of variable we obtain
\begin{align*}
\int_{\rn}|U_{\delta,\xi}^{p-1-\eps}-U_{\delta,\xi}^{p-1}|^\frac{n}{2}
&=\int_{\rn}|(\delta^{-\frac{n-2}{2}}U_{1,0})^{p-1-\eps}-(\delta^{-\frac{n-2}{2}}U_{1,0})^{p-1}|^\frac{n}{2} \delta^n\\
  &=\int_{\rn}|\delta^{\eps\frac{n-2}{2}} U_{1,0}^{p-1-\eps} -U_{1,0}^{p-1}|^\frac{n}{2}
  =\int_{\rn}|
  \delta^{\eps\frac{n-2}{2}} (U_{1,0}^{p-1-\eps}
  -U_{1,0}^{p-1})
  +(\delta^{\eps\frac{n-2}{2}}-1)U_{1,0}^{p-1}|^\frac{n}{2}.
\end{align*}
Therefore, by  \eqref{para} and \eqref{re1},
\begin{align}
|f'_\eps(U_{\delta,\xi})-f'_0(U_{\delta,\xi})|_\frac{n}{2}
&\leq C(\,|U_{\delta,\xi}^{p-1-\eps}-U_{\delta,\xi}^{p-1}|_\frac{n}{2} + \eps|U_{\delta,\xi}^{p-1}|_\frac{n}{2})\notag \\
&\leq C(\,\delta^{\eps\frac{n-2}{2}}\,|f'_\eps(U_{1,0}) - f'_0(U_{1,0})|_\frac{n}{2} + |\delta^{\eps\frac{n-2}{2}}-1|\,|U_{1,0}|_{2^*}^{p-1} + \eps|U_{\delta,\xi}^{p-1}|_\frac{n}{2}) \notag\\
&\leq C(|\eps^{\eps\frac{n-2}{2n}}|\eps +|\eps^{\eps\frac{n-2}{2n}}-1|+\eps)\leq C\eps|\ln \eps|
\label{re4}
\end{align}
(see \cite[Lemma B.13]{PST23} for a different argument).  An easy computation gives
\begin{align}\label{re2}
|{\mathds 1}_{\rn\setminus\o}f'_0(U_{\delta,\xi})|_\frac{n}{2}\leq C\delta^2.
\end{align}
Now Lemmas~\ref{lem:holder} and~\ref{lem:projections} yield the result.

$(iii):$ \ Using the mean value theorem as in \cite[Remark 2.21]{mp} one shows that $f_\eps(U_{1,0})-f_0(U_{1,0})\in L^\frac{2n}{n+2}(\rn)\cap L^\frac{sn}{n+2s}(\rn)$ for $\eps$ small enough, and that
\begin{equation}\label{mpineq}
|f_\eps(U_{1,0})-f_0(U_{1,0})|_\frac{2n}{n+2}\leq C\eps\qquad\text{and}\qquad|f_\eps(U_{1,0})-f_0(U_{1,0})|_\frac{sn}{n+2s}\leq C\eps.
\end{equation}
Then, by scaling,
\begin{align*}
\int_{\rn}|U_{\delta,\xi}^{p-\eps}-U_{\delta,\xi}^{p}|^\frac{2n}{n+2}
&=\int_{\rn}|(\delta^{-\frac{n-2}{2}}U_{1,0})^{p-\eps}-(\delta^{-\frac{n-2}{2}}U_{1,0})^{p}|^\frac{2n}{n+2} \delta^n\\
  &=\int_{\rn}|\delta^{\eps\frac{n-2}{2}} U_{1,0}^{p-\eps} -U_{1,0}^{p}|^\frac{2n}{n+2}
  =\int_{\rn}|
  \delta^{\eps\frac{n-2}{2}} (U_{1,0}^{p-\eps}
  -U_{1,0}^{p})
  +(\delta^{\eps\frac{n-2}{2}}-1)U_{1,0}^{p}|^\frac{2n}{n+2}.
\end{align*}
Hence, using \eqref{para} and \eqref{mpineq},
\begin{align}
|f_\eps(U_{\delta,\xi})-f_0(U_{\delta,\xi})|_\frac{2n}{n+2}&\leq |\delta^{\eps\frac{n-2}{2}}|\,|U_{1,0}^{p-\eps}
  -U_{1,0}^{p}|_\frac{2n}{n+2}
  +|\delta^{\eps\frac{n-2}{2}}-1|\,|U_{1,0}|^p_{2^*} \notag\\
&\leq C(|\eps^{\eps\frac{n-2}{2n}}|\eps
  +|\eps^{\eps\frac{n-2}{2n}}-1|\leq C\eps|\ln \eps|. \label{eq:iii}
\end{align}
One can argue similarly for $|f_\eps(U_{\delta,\xi})-f_0(U_{\delta,\xi})|_\frac{sn}{n+2s}$ to obtain
\begin{equation}\label{lol}
|f_\eps(U_{\delta,\xi})-f_0(U_{\delta,\xi})|_\frac{sn}{n+2s}\leq C\eps|\ln \eps|
\end{equation}
(see also \cite[Lemma B.12]{PST23} for a different argument).  Since
\begin{align}
|{\mathds 1}_{\rn\setminus\o}f_0(U_{\delta,\xi})|_\frac{2n}{n+2}\leq C\delta^\frac{n+2}{2}\qquad\text{and}\qquad |{\mathds 1}_{\rn\setminus\o}f_0(U_{\delta,\xi})|_\frac{sn}{n+2s}\leq C\delta^\frac{n+2}{2},\label{eq:mp2}
\end{align}
the statement follows from Lemmas~\ref{lem:projections} and ~\ref{lem:holder}.
\end{proof}

\begin{remark}
\emph{Note that some functions in the previous lemma depend on $d,\xi,$ and $\eps$, but we have omitted these indices to ease the notation.}
\end{remark}

The following proposition establishes the existence of the error $\phi_{\delta,\xi}$ and a bound for its norm.

\begin{proposition}\label{rid1}
Given an open bounded set $\t$ with $\overline{\t}\subset(0,\infty)\times\Omega$, there exist $c>0$ and $\eps_0>0$ such that, for each $\eps\in(0,\eps_0)$ and $(d,\xi)\in \t$, there is a unique $\phi_{\delta,\xi}\in K^\perp_{\delta,\xi}\cap L^s(\rn)$ with $\delta=d\eps^\frac{1}{n}$ which satisfies
\begin{align}\label{phi:err}
\|\phi_{\delta,\xi}\|_E\leq c\eps^\frac{n+2}{2n}
\end{align}
and solves the first equation in \eqref{s}. Furthermore, for each $\eps\in(0,\eps_0)$, the map $(d,\xi)\mapsto \phi_{\delta,\xi}$ is of class $\cC^1$ in $\t$.
\end{proposition}

\begin{proof}
Let $\mathscr L:K^\perp_{\delta,\xi}\cap L^s(\rn)\to K^\perp_{\delta,\xi}\cap L^s(\rn)$ be given by
$$\mathscr L(\phi):=\Pi_{\delta,\xi}^\perp[\phi-i^*[f'_0 (U_{\delta,\xi})\phi]].$$
Since $s\in(\frac{2n}{n+2},2^*)$, this is a Banach isomorphism, i.e., $\mathscr L$ is invertible and there is a constant $C=C(n,s)$ such that, for $\eps>0$ small enough,
$$\|\mathscr L^{-1}(\psi)\|_E\leq C\|\psi\|_E\qquad\text{for every \ }\psi\in K^\perp_{\delta,\xi}\cap L^s(\rn), \ (d,\xi)\in\overline{\t}, \ \delta=d\eps^\frac{1}{n},$$
see Lemmas~\ref{B1} and~\ref{B2}. Noting that $U_{\delta,\xi}=i^* \left[ f_0 \left(U_{\delta,\xi}\right)\right]$ and that $Q_\o v=v-2{\mathds 1}_{\mathbb R^n\setminus\Omega}v$, the first equation in \eqref{s} can be rewritten as
\begin{align*}
\mathscr L(\phi):=\Pi_{\delta,\xi}^\perp\left[\phi-i^* \left[   f'_0 \left(U_{\delta,\xi} \right)\phi \right]\right]
 &=\underbrace{(\Pi_{\delta,\xi}^\perp\circ i^*)\left[  Q_\Omega\left[  f_\eps \left(U_{\delta,\xi}+\phi \right) -
 f_\eps \left(U_{\delta,\xi}  \right) - f'_\eps \left(U_{\delta,\xi}\right)\phi  \right]\right]}_{=:\mathscr N(\phi)} \notag\\
&+\underbrace{(\Pi_{\delta,\xi}^\perp\circ i^*)\left[  Q_\Omega\left[  f'_\eps \left(U_{\delta,\xi}  \right)
 - f'_0\left(U_{\delta,\xi}\right)  \right]\phi\right]
 +2(\Pi_{\delta,\xi}^\perp\circ i^*)\left[-{\mathds 1}_{\mathbb R^n\setminus\Omega}  f'_0\left(U_{\delta,\xi}\right) \phi\right]}_{=:\mathscr L_\eps (\phi)}\notag\\
&+\underbrace{(\Pi_{\delta,\xi}^\perp\circ i^*)\left[
Q_\Omega\left[  f_\eps \left(U_{\delta,\xi}  \right)
 - f_0\left(U_{\delta,\xi}\right)  \right]\right]
 +2(\Pi_{\delta,\xi}^\perp\circ i^*)\left[
 -{\mathds 1}_{\mathbb R^n\setminus\Omega}  f_0\left(U_{\delta,\xi}\right)\right].}_{=:\mathscr E}
\end{align*}
Let $T=T_{\delta,\xi}:K^\perp_{\delta,\xi}\cap L^s(\rn)\to K^\perp_{\delta,\xi}\cap L^s(\rn)$ be given by
\begin{align*}
 T(\phi) := \mathscr L^{-1}\left(\mathscr N(\phi)+\mathscr L_\eps (\phi)+\mathscr E\right).
\end{align*}
It follows from Lemma~\ref{lem:estimates} that
\begin{align*}
 \|T \phi\|_E \leq C(\|\mathscr N(\phi)\|_E+\|\mathscr L_\eps(\phi)\|_E+\|\mathscr E\|_E)\leq C(\|\phi\|_E^2+\|\phi\|_E^{p-\eps}+(\eps|\ln \eps|+\delta^2)\|\phi\|_E+\eps|\ln \eps|+\delta^{\frac{n+2}{2}}).
\end{align*}
 So, for a suitable constant $c>0$ and $\eps$ small enough, the map
\begin{align*}
T:\{\phi \in K^\perp_{\delta,\xi}\cap L^s(\rn):\|\phi\|_{E} \leq c\eps^\frac{n+2}{2n}\} \to\{\phi \in K^\perp_{\delta,\xi}\cap L^s(\rn):\|\phi\|_E \leq c\eps^\frac{n+2}{2n}\}
\end{align*}
is well defined. Applying the mean value theorem to the function $t\mapsto f_\eps(U_{\delta,\xi}+t\phi_1+(1-t)\phi_2)$ and arguing as in the proof of Lemma~\ref{lem:estimates}$(i)$ we get
\begin{align*}
 \|T\phi_1-T\phi_2\|_E & \leq C( \|\mathscr N(\phi_1)-\mathscr N(\phi_2)\|_E + \|\mathscr L_\eps (\phi_1-\phi_2)\|_E)\\
&\leq C(\|\phi_1\|_E + \|\phi_2\|_E + \|\phi_1\|_E^{p-\eps-1} + \|\phi_2\|_E^{p-\eps-1})\|\phi_1-\phi_2\|_E +(\eps|\ln \eps|+\delta^2)\|\phi_1-\phi_2\|_E.
\end{align*}
Hence, $T$ is a contraction mapping for $\eps$ small enough and, by Banach's fixed point theorem, there exists a unique $\phi_\eps^{d,\xi}\in K^\perp_{\delta,\xi}\cap L^s(\rn)$ with $\|\phi_\eps^{d,\xi}\|_{E} \leq c\eps^\frac{n+2}{2n}$ satisfying $T(\phi_\eps^{d,\xi})=\phi_\eps^{d,\xi}$, i.e., $\mathscr L(\phi_\eps^{d,\xi})=\mathscr N(\phi_\eps^{d,\xi})+\mathscr L_\eps (\phi_\eps^{d,\xi})+\mathscr E$ or, equivalently, $\phi_\eps^{d,\xi}$ solves the first equation in \eqref{s}.

Next, we prove that the map $(d,\xi)\mapsto \phi_{\delta,\xi}$ is of class $\cC^1$ in $\t$ for small enough $\eps$. To this end, we fix such an $\eps>0$ and set $\cU_\eps:=\{(d,\xi,\psi)\in \t\times [K_{1,0}^\perp\cap L^s(\rn)]:\|P_{\delta,\xi}\psi\|_{E} \leq c\eps^\frac{n+2}{2n}\}$ with $\delta=d\eps^\frac{1}{n}$ and $P_{\delta,\xi}$ as in \eqref{eq:rescaling}. The map $F_\eps:\cU_\eps\to K_{0,1}^\perp\cap L^s(\rn)$ given by
$$F_\eps(d,\xi,\psi)=\psi-P_{\delta,\xi}^{-1}T_{\delta,\xi}(P_{\delta,\xi}\psi)$$
is of class $\cC^1$, \ $F_\eps^{-1}(0)=\{(d,\xi,P_{\delta,\xi}^{-1}\phi_\eps^{d,\xi}):(d,\xi)\in\t\}$ \ and
\begin{align*}
\partial_\psi F_\eps(d,\xi,\psi)=\mathrm{id}_{K_{0,1}^\perp\cap L^s(\rn)} - P_{\delta,\xi}^{-1}\circ T'_{\delta,\xi}(P_{\delta,\xi}\psi)\circ P_{\delta,\xi}=P_{\delta,\xi}^{-1}\circ \big(\mathrm{id}_{K_{\delta,\xi}^\perp\cap L^s(\rn)} - T'_{\delta,\xi}(P_{\delta,\xi}\psi)\big)\circ P_{\delta,\xi}.
\end{align*}
We claim that $\partial_\psi F_\eps(d,\xi,P_{\delta,\xi}^{-1}\phi_\eps^{d,\xi}):K^\perp_{1,0}\cap L^s(\rn)\to K^\perp_{1,0}\cap L^s(\rn)$ is a linear homeomorphism for each $(d,\xi)\in\t$ if $\eps$ is sufficiently small. Indeed, since
\begin{align*}
T_{\delta,\xi}'(\phi)[\psi]&=\mathscr L^{-1}[\mathscr N'(\phi)[\psi]-\mathscr L_\eps\psi]=\mathscr L^{-1}[(\Pi^\perp_{\delta,\xi}\circ i^*)\big[Q_\o(f'_\eps(U_{\delta,\xi}+\phi)-f'_\eps(U_{\delta,\xi}))\psi]\big] -\mathscr L^{-1}\mathscr L_\eps\psi
\end{align*}
for all $\phi,\psi\in K^\perp_{\delta,\xi}\cap L^s(\rn)$, arguing as in the proof of Lemma~\ref{lem:estimates} we obtain
$$\|T_{\delta,\xi}'(\phi_\eps^{d,\xi})[\psi]\|_E\leq C\big[\|\phi_\eps^{d,\xi}\|_E+\|\phi_\eps^{d,\xi}\|_E^{p-\eps-1} + (\eps|\ln \eps|+\delta^2)\big]\|\psi\|_E.$$
This shows that \ $T'_{\delta,\xi}(\phi_\eps^{d,\xi})$ \ tends to $0$ as $\eps\to 0$ in the Banach space space of continuous linear maps $K^\perp_{\delta,\xi}\cap L^s(\rn)\to K^\perp_{\delta,\xi}\cap L^s(\rn)$. Since the set of linear homeomorphisms is open in this space (see \cite[(8.3.2)]{d}), we have that $\mathrm{id}_{K_{\delta,\xi}^\perp\cap L^s(\rn)} - T'_{\delta,\xi}(\phi_\eps^{d,\xi})$ is a linear homeomorphism for $\eps$ sufficiently small. Hence, $\partial_\psi F_\eps(d,\xi,P_{\delta,\xi}^{-1}\phi_\eps^{d,\xi})$ is a linear homeomorphism for $\eps$ sufficiently small, as claimed. Now we may apply the implicit function theorem \cite[(10.2.1)]{d} to derive that $(d,\xi)\mapsto P_{d,\xi}^{-1}\phi_{\delta,\xi}$ is of class $\cC^1$. It follows that $(d,\xi)\mapsto \phi_{\delta,\xi}$ is of class $\cC^1$. This completes the proof.
\end{proof}

\subsection{Finding the right parameters}

Our next goal is to show that there are parameters $\delta$ and $\xi$ that solve the second equation in \eqref{s} with $\phi=\phi_{\delta,\xi}$. They are related to the critical points of the following function.

Given an open bounded set $\t$ with $\overline{\t}\subset(0,\infty)\times\Omega$, let $\eps_0>0$ be as in Proposition~\ref{rid1} and, for $\eps\in(0,\eps_0)$, define $\tilde J_\eps:\t\to\r$ by
\begin{align*}
\tilde J_\eps(d,\xi):=J_\eps(U_{\delta,\xi}+\phi_{\delta,\xi}),
\end{align*}
 where $\delta=d\eps^\frac{1}{n}$ and
 \begin{align*}
J_\eps(u):=\frac{1}{2} \irn|\nabla u|^2-\frac {1}{p+1-\eps}\irn Q_\o |u|^{p+1-\eps}.
 \end{align*}
 $\tilde J_\eps$ is called \emph{the reduced energy functional}.  The following lemma shows its relevance.

\begin{lemma}\label{iff:lem}
Let $(d,\xi)\in\t$ and let $\phi_{\delta,\xi}$ be as in Proposition \ref{rid1}. Then, for $\eps$ small enough,
\begin{align*}
\text{$U_{\delta,\xi}+\phi_{\delta,\xi}$ solves \eqref{s} if and only if $(d,\xi)$ is a critical point of $\tilde J_\eps$.}
\end{align*}
\end{lemma}

\begin{proof}
Fix $\eps>0$, small enough. Note that $U_{\delta,\xi}+\phi_{\delta,\xi}$ solves \eqref{s} if and only if $U_{\delta,\xi}+\phi_{\delta,\xi}$ is a critical point of $J_\eps$.

The function $h_\eps:\t\to E$ given by $h_\eps(d,\xi):=U_{\delta,\xi}+\phi_{\delta,\xi}$ with $\delta=d\eps^\frac{1}{n}$ is of class $\cC^1$. By the chain rule, $\widetilde J'_\eps(d,\xi)=J'_\eps(U_{\delta,\xi}+\phi_{\delta,\xi})\circ h'_\eps(d,\xi)$. Therefore, if $U_{\delta,\xi}+\phi_{\delta,\xi}$ is a critical point of $J_\eps$, then $(d,\xi)$ is a critical point of $\tilde J_\eps$.

On the other hand, as $U_{\delta,\xi}+\phi_{\delta,\xi}$ solves the first equation in \eqref{s}, we get that
$$U_{\delta,\xi}+\phi_{\delta,\xi}-i^*[Q_\o f_\eps(U_{\delta,\xi}+\phi_{\delta,\xi})]\in K_{\delta,\xi}.$$
Therefore,
$$\langle U_{\delta,\xi}+\phi_{\delta,\xi},\psi\rangle = \langle i^*(Q_\o f_\eps(U_{\delta,\xi}+\phi_{\delta,\xi}),\psi\rangle = \irn Q_\o f_\eps(U_{\delta,\xi}+\phi_{\delta,\xi}) \psi\qquad\text{for every \ }\psi\in K_{\delta,\xi}^\perp\cap L^s(\rn),$$
that is, 
\begin{equation}\label{eq:J'K perp}
J'_\eps(U_{\delta,\xi}+\phi_{\delta,\xi})[\psi]=0\qquad\text{for every \ }\psi\in K_{\delta,\xi}^\perp\cap L^s(\rn).
\end{equation}
If $(d,\xi)$ is a critical point of $\tilde J_\eps$, then, recalling that $Z^i_{\delta,\xi}=\delta\, \partial_{\xi_i}U_{\delta,\xi}$ and using \eqref{eq:J'K perp} we have that
\begin{equation}\label{eq:kernel}
0=\delta\, J'_\eps(U_{\delta,\xi}+\phi_{\delta,\xi})\,[\partial_{\xi_i}U_{\delta,\xi} + \partial_{\xi_i}\phi_{\delta,\xi}]=J'_\eps(U_{\delta,\xi}+\phi_{\delta,\xi})\,[Z^i_{\delta,\xi} + \delta\,\Pi_{\delta,\xi}\partial_{\xi_i}\phi_{\delta,\xi}],\quad i=1,\ldots,n,
\end{equation}
and similarly for $\partial_{\delta}$. Our goal is to show that
\begin{equation}\label{eq:claim}
K_{\delta,\xi}=\mathrm{span}\{Z^0_{\delta,\xi} + \delta\,\Pi_{\delta,\xi}\partial_{\delta}\phi_{\delta,\xi}, \ Z^i_{\delta,\xi} + \delta\,\Pi_{\delta,\xi}\partial_{\xi_i}\phi_{\delta,\xi},\, i=1,\ldots,n\},
\end{equation}
as, then, it would follow from \eqref{eq:kernel} that
\begin{equation*}
J'_\eps(U_{\delta,\xi}+\phi_{\delta,\xi})[Z]=0\qquad\text{for every \ }Z\in K_{\delta,\xi}.
\end{equation*}
This, together with \eqref{eq:J'K perp} implies that $U_{\delta,\xi}+\phi_{\delta,\xi}$ is a critical point of $J_\eps$, as claimed.

To prove \eqref{eq:claim} recall that $K_{\delta,\xi}=\mathrm{span}\{Z^i_{\delta,\xi}:i=0,1,\ldots,n\}$. Note that $\langle Z^i_{\delta,\xi},Z^j_{\delta,\xi}\rangle =0$ if $i\neq j$, and that $\|Z^i_{\delta,\xi}\|=\|Z^i_{1,0}\|$ does not depend on $\delta$ and $\xi$. As $Z^i_{\delta,\xi}$ solves \eqref{eq:linear equation}, we have that $-\Delta(\partial_{\xi_j}Z^i_{\delta,\xi})=p\partial_{\xi_j}(U_{\delta,\xi}^{p-1}Z^i_{\delta,\xi})$. Testing this equation with $\partial_{\xi_j}Z^i_{\delta,\xi}$ yields that $\|\partial_{\xi_j}Z^i_{\delta,\xi}\|=O(\delta^{-1})$. Now,
$$\Pi_{\delta,\xi}\partial_{\xi_i}\phi_{\delta,\xi}=\sum_{j=0}^n\langle Z^j_{\delta,\xi},\partial_{\xi_i}\phi_{\delta,\xi}\rangle\frac{Z^j_{\delta,\xi}}{\|Z^j_{\delta,\xi}\|^2}.$$
Since $\langle Z^j_{\delta,\xi},\phi_{\delta,\xi}\rangle=0$, we have that $\langle Z^j_{\delta,\xi},\partial_{\xi_i}\phi_{\delta,\xi}\rangle=-\langle \partial_{\xi_i}Z^j_{\delta,\xi},\phi_{\delta,\xi}\rangle$. Therefore,
\begin{equation}\label{eq:norm phi}
\|\delta\,\Pi_{\delta,\xi}\partial_{\xi_i}\phi_{\delta,\xi}\|\leq C\|\phi_{\delta,\xi}\|\leq C\eps^\frac{n+2}{2n},
\end{equation}
and, similarly, for $\partial_\delta$. As a consequence, $Z^0_{\delta,\xi} + \delta\,\Pi_{\delta,\xi}\partial_{\delta}\phi_{\delta,\xi}, \ Z^1_{\delta,\xi} + \delta\,\Pi_{\delta,\xi}\partial_{\xi_1}\phi_{\delta,\xi}, \ \ldots, \ Z^n_{\delta,\xi} + \delta\,\Pi_{\delta,\xi}\partial_{\xi_n}\phi_{\delta,\xi}$ are linearly independent for $\eps$ small enough. This proves \eqref{eq:claim}, and completes the proof.
\end{proof}

To find suitable critical points of $\widetilde J$, we need to identify its \emph{leading terms}. This is the point of the following proposition.

\begin{proposition}
\label{red-en}
The expansion
\begin{align}\label{em1}
\tilde J_\eps(d,\xi)=\mathfrak a+\mathfrak b \eps+\mathfrak c \eps\ln\eps+\eps\Psi(d,\xi)+o(\eps),
\end{align}
holds true $\cC^1$-uniformly in $\overline{\t}$ as $\eps\to 0$, where
\begin{equation}\label{PSI}
\Psi(d,\xi):=\mathfrak c_1 d^n
\psi_\Omega(\xi)
-\mathfrak c_2\ln d,\qquad
\psi_\Omega(\xi):=\int_{\mathbb R^n\setminus \Omega}\frac1{|x-\xi|^{2n}}\d x.
\end{equation}
All constants depend only on $n$, $\mathfrak a=\frac{1}{n}S^\frac{n}{2}$, where $S$ is the best constant for the Sobolev embedding $D^{1,2}(\rn)\hookrightarrow L^{2^*}(\rn)$, and the constants $\mathfrak c_1$ and $\mathfrak c_2$ are positive.
\end{proposition}

\begin{proof}
Set $\phi:=\phi_{\delta,\xi}$. As $U_{\delta,\xi}$ solves \eqref{eq:yamabe} we have that
 \begin{align} \label{eq:J1}
|J_\eps(U_{\delta,\xi}+\phi)-J_\eps(U_{\delta,\xi})|&=\frac{1}{2}\|\phi\|^2+\irn\nabla U_{\delta,\xi}\cdot\nabla\phi-\frac{1}{p+1-\eps}\irn Q_\o\Big(|U_{\delta,\xi}+\phi|^{p+1-\eps}-U_{\delta,\xi}^{p+1-\eps}\Big)\nonumber \\
&\leq \frac{1}{2}\|\phi\|^2+\irn(1-Q_\o)U_{\delta,\xi}^p\phi + \irn Q_\o(U_{\delta,\xi}^p - U_{\delta,\xi}^{p-\eps})\phi \nonumber\\
&\qquad -\frac{1}{p+1-\eps}\irn Q_\o\Big(|U_{\delta,\xi}+\phi|^{p+1-\eps}-U_{\delta,\xi}^{p+1-\eps}-(p+1-\eps)U_{\delta,\xi}^{p-\eps}\phi\Big).
\end{align}
Since $\|\phi\|_E\leq c\eps^\frac{n+2}{2n}$ and $|{\mathds 1}_{\mathbb R^n\setminus\Omega}U_{\delta,\xi}|_{p+1}=C\delta^\frac{n-2}{2}=C\eps^\frac{n-2}{2n}$, Hölder's inequality yields
\begin{equation}\label{eq:J2}
\irn (1-Q_\Omega)U_{\delta,\xi}^p|\phi|\leq C|{\mathds 1}_{\mathbb R^n\setminus\Omega}U_{\delta,\xi}|_{p+1}^p|\phi|_{p+1}\leq C\eps^\frac{n+2}{n}.
\end{equation}
Using \eqref{eq:iii} and Hölder's inequality,
\begin{equation}\label{eq:J3}
\irn |U_{\delta,\xi}^p - U_{\delta,\xi}^{p-\eps}||\phi| \leq |U_{\delta,\xi}^p - U_{\delta,\xi}^{p-\eps}|_\frac{2n}{n+2}|\phi|_{p+1}\leq C\eps|\ln\eps||\phi|_{p+1}\leq C\eps^\frac{n+2}{n}.
\end{equation}
Now, by the mean value theorem there is $t\in(0,1)$ such that
\begin{align*}
&\frac{1}{p+1-\eps}\Big||U_{\delta,\xi}+\phi|^{p+1-\eps}-U_{\delta,\xi}^{p+1-\eps}-(p+1-\eps)U_{\delta,\xi}^{p-\eps}\phi\Big|=\Big||U_{\delta,\xi}+t\phi|^{p-1-\eps}(U_{\delta,\xi}+t\phi)-U_{\delta,\xi}^{p-\eps}\Big|\,|\phi| \\
&\qquad\leq\Big||U_{\delta,\xi}+t\phi|^{p-1-\eps}(U_{\delta,\xi}+t\phi)-U_{\delta,\xi}^{p-\eps}-(p-\eps)U_{\delta,\xi}^{p-1-\eps}t\phi\Big|\,|\phi| + (p-\eps)U_{\delta,\xi}^{p-1-\eps}t\phi^2
\end{align*}
and, by Lemma~\ref{lem:li},
\begin{align*}
\big||U_{\delta,\xi}+t\phi|^{p-1-\eps}(U_{\delta,\xi}+t\phi)-U_{\delta,\xi}^{p-\eps}-(p-\eps)U_{\delta,\xi}^{p-1-\eps}t\phi\big|\leq
\begin{cases}
C|\phi|^{p-\eps} &\text{if \ }1\leq p-\eps<2,\\
C(|\phi|^{p-\eps}+U_{\delta,\xi}^{p-2-\eps}\phi^2) & \text{if \ }2\leq p-\eps.
\end{cases}
\end{align*}
So, Hölder's inequality yields that
\begin{align}\label{eq:J4}
&\frac{1}{p+1-\eps}\irn\Big||U_{\delta,\xi}+\phi|^{p+1-\eps}-U_{\delta,\xi}^{p+1-\eps}-(p+1-\eps)U_{\delta,\xi}^{p-\eps}\phi\Big|\nonumber\\
&\qquad\leq 
\begin{cases}
C\irn \Big(|\phi|^{p+1-\eps} + U_{\delta,\xi}^{p-1-\eps}|\phi|^2\Big)\leq C(\|\phi\|_E^{p+1-\eps}+\|\phi\|_E^2)&\text{if \ }1\leq p-\eps<2,\\
C\irn \Big(|\phi|^{p+1-\eps} + U_{\delta,\xi}^{p-2-\eps}|\phi|^3 + U_{\delta,\xi}^{p-1-\eps}|\phi|^2\Big)\leq C(\|\phi\|_E^{p+1-\eps}+\|\phi\|_E^3+\|\phi\|_E^2)&\text{if \ }2\leq p-\eps,
\end{cases} \nonumber\\
&\qquad\leq C\eps^\frac{n+2}{n}
\end{align}
for $\eps$ small enough. It follows from \eqref{eq:J1}, \eqref{eq:J2}, \eqref{eq:J3}, and \eqref{eq:J4} that
\begin{equation} \label{ee1}
J_\eps(U_{\delta,\xi}+\phi)=J_\eps(U_{\delta,\xi})+ O(\eps^{\frac{n+2}{n}}).
\end{equation}
Next, we give an expansion for $J_\eps\left(U_{\delta,\xi} \right)$. Note that
 \begin{align} \label{ee2}
 J_\eps(U_{\delta,\xi})&=J_0(U_{\delta,\xi}) + \frac{1}{p+1}\irn|U_{\delta,\xi}|^{p+1}-\frac{1}{p+1-\eps}\irn |U_{\delta,\xi}|^{p+1-\eps} + \frac{1}{p+1-\eps}\irn (1-Q_\o)|U_{\delta,\xi}|^{p+1-\eps} \notag\\ 
 &=\underbrace{J_0(U_{\delta,\xi})}_{=:\mathfrak a} + \frac{1}{p+1}\irn|U_{\delta,\xi}|^{p+1}-\frac{1}{p+1-\eps}\irn |U_{\delta,\xi}|^{p+1-\eps} + \frac{2}{p+1}\int_{\rn\setminus\o}|U_{\delta,\xi}|^{p+1} \notag \\
 &\qquad + \frac{2}{p+1-\eps}\int_{\rn\setminus\o}|U_{\delta,\xi}|^{p+1-\eps} - \frac{2}{p+1}\int_{\rn\setminus\o}|U_{\delta,\xi}|^{p+1}.
\end{align}
Using a Taylor expansion we write
\begin{align} \label{ee3}
\frac1{p+1-\eps}\int_{\mathbb R^n}| U_{\delta,\xi}|^{p+1-\eps}&=\frac1{p+1}\int_{\mathbb R^n}| U_{\delta,\xi}|^{p+1}+\eps\left(\frac1{(p+1)^2}\int_{\mathbb R^n}| U_{\delta,\xi}|^{p+1}-\frac1{p+1}\int_{\mathbb R^n}| U_{\delta,\xi}|^{p+1}\ln U_{\delta,\xi}\right)+O(\eps^2).
\end{align}
Setting \ $x-\xi=\delta y$ \ and recalling that \ $\delta=d\eps^\frac{1}{n}$,
\begin{align*}
&\frac1{(p+1)^2}\int_{\mathbb R^n}| U_{\delta,\xi}|^{p+1}-\frac1{p+1}\int_{\mathbb R^n}| U_{\delta,\xi}|^{p+1}\ln U_{\delta,\xi} = \frac1{(p+1)^2}\int_{\mathbb R^n}  U^{p+1}- \frac1{p+1}\int_{\mathbb R^n} U^{p+1}\left(\ln U-\frac{n-2}2\ln\delta\right)\\
&\qquad=\underbrace{\frac1{(p+1)^2}\int_{\mathbb R^n}  U^{p+1}- \frac1{p+1}\int_{\mathbb R^n} U^{p+1} \ln U}_{=:-\mathfrak b}+\underbrace{\frac{1}{(p+1)^2}\int_{\mathbb R^n} U^{p+1}}_{=:-\mathfrak c} \ln \eps+\underbrace{\frac{n}{(p+1)^2} \int_{\mathbb R^n} U^{p+1}}_{=:\mathfrak c_2} \ln d.
\end{align*}
Hence,
\begin{align}\label{ee4}
\frac{1}{p+1}\irn|U_{\delta,\xi}|^{p+1}-\frac{1}{p+1-\eps}\irn |U_{\delta,\xi}|^{p+1-\eps} =\eps\left(\mathfrak b + \mathfrak c\ln \eps -\mathfrak c_2\ln d\right)+O(\eps^2).
\end{align}
A straightforward computation shows that
 \begin{align}\label{ee5}
\frac{2}{p+1}\int_{\rn\setminus\o}|U_{\delta,\xi}|^{p+1}=\delta^n\Big[\frac{2\alpha_n^{p+1}}{p+1}\int_{\mathbb R^n\setminus\Omega}\frac 1{|x-\xi|^{2n}}\d x+o(1)\Big]=\eps d^n\underbrace{\frac{2\alpha_n^{p+1}}{p+1}}_{=:\mathfrak c_1}\int_{\mathbb R^n\setminus\Omega}\frac{1}{|x-\xi|^{2n}}\d x + o(\eps).
 \end{align}
Finally, a Taylor expansion as in \eqref{ee3} and a simple computation yield
\begin{align}\label{ee6}
\frac2{p+1-\eps}\int_{\rn\setminus\o}| U_{\delta,\xi}|^{p+1-\eps}-\frac2{p+1}\int_{\rn\setminus\o}| U_{\delta,\xi}|^{p+1}=\eps\, o(1)=o(\eps).
\end{align}
From \eqref{ee1}, \eqref{ee2}, \eqref{ee4}, \eqref{ee5}, and \eqref{ee6} we derive
\begin{align*}
\tilde J_\eps(d,\xi)=\mathfrak a+\mathfrak b \eps+\mathfrak c \eps\ln\eps+\eps\Psi(d,\xi)+o(\eps),
\end{align*}
for all $(d,\xi)\in\t$ and $\eps>0$ small enough, with
\begin{equation*}
\Psi(d,\xi):=\mathfrak c_1 d^n\left(\ \int_{\mathbb R^n\setminus \Omega}\frac1{|x-\xi|^{2n}}\d x\right)-\mathfrak c_2\ln d,
\end{equation*}
$\cC^0$-uniformly in $\overline{\t}$. 

To show that the expansion holds true $\cC^1$-uniformly in $\overline{\t}$, note that, by \eqref{eq:J'K perp},
\begin{align*}
\partial_{\xi_i}J_\eps(U_{\delta,\xi} + \phi_{\delta,\xi})-\partial_{\xi_i}J_\eps(U_{\delta,\xi})=(J'_\eps(U_{\delta,\xi} + \phi_{\delta,\xi}) - J'_\eps(U_{\delta,\xi}))[\partial_{\xi_i}U_{\delta,\xi}] + J'_\eps(U_{\delta,\xi} + \phi_{\delta,\xi})[\Pi_{\delta,\xi}\partial_{\xi_i}\phi_{\delta,\xi}].
\end{align*}
Now, using \eqref{eq:norm phi}, one can estimate this difference and the similar one for $\partial_{\delta}$; see, for instance, \cite[Section 7]{gmp}.
\end{proof}

Finally, to guarantee that the solutions that we find do not change sign, we use a variational argument.  The nontrivial solutions to the problem \eqref{eq:problem+} are the critical points of the functional
$$J_q(u):=\frac{1}{2}\|u\|^2-\frac{1}{q+1}\irn Q_\o|u|^{q+1}$$
on the Nehari manifold
$$\cN_q:=\Big\{u\in E:u\neq 0, \ \|u\|^2=\irn Q_\o|u|^{q+1}\Big\}.$$
Set
$$c_q:=\inf_{u\in\cN_q}J_q(u)=\inf_{u\in\cN_q}\frac{q-1}{2(q+1)}\|u\|^2.$$

\begin{lemma}\hspace{2em}\label{lem:positive solution}
\begin{itemize}
\item[$(i)$] $\lim_{q\to p}c_q=\frac{1}{n}S^\frac{n}{2}$, where $S$ is the best constant for the Sobolev embedding $D^{1,2}(\rn)\hookrightarrow L^{2^*}(\rn)$.
\item[$(ii)$] If $u$ is a nontrivial solution to \eqref{eq:problem+} and $J_q(u)<2c_q$, then $u$ is strictly positive or strictly negative in $\rn$.
\end{itemize}
\end{lemma}

\begin{proof}
$(i):$ \ If $u\in\cN_q$ then, using Hölder's and Sobolev's inequalities, we get
\begin{align*}
\|u\|^2&=\irn Q_\o|u|^{q+1}\leq \io|u|^{q+1}\leq |\o|^\frac{p-q}{p+1}\Big(\io|u|^{p+1}\Big)^\frac{q+1}{p+1}\leq |\o|^\frac{p-q}{p+1}\Big(\irn|u|^{p+1}\Big)^\frac{q+1}{p+1} \leq |\o|^\frac{p-q}{p+1}S^{-\frac{q+1}{2}}\|u\|^{q+1}.
\end{align*}
Therefore,
$$\frac{q-1}{2(q+1)}S^\frac{q+1}{q-1}|\o|^{-\frac{p-q}{p+1}\frac{2}{q-1}}\leq \frac{q-1}{2(q+1)} \|u\|^2=J_q(u)\qquad\text{for every \ }u\in\cN_q.$$
As a consequence, $\frac{1}{n}S^\frac{n}{2}\leq \liminf\limits_{p\to q}c_q.$ Next, observe that 
$$c_q=\inf_{u\in\cN_q}J_q(u)\leq\inf_{u\in\cN_q\cap D^{1,2}_0(\o)}J_q(u)=:c_{q,\o}.$$
Note that $c_{q,\o}$ is the least energy of a nontrivial solution to the Dirichlet problem $-\Delta u=|u|^{q-1}u$, \ $u\in D^{1,2}_0(\o)$. It is well known that $\lim_{q\to p}c_{q,\o}=\frac{1}{n}S^\frac{n}{2}$; see \cite[Lemma 4.1]{bc}. Therefore, $\limsup\limits_{p\to q}c_q\leq \frac{1}{n}S^\frac{n}{2}$. This proves $(i)$.

$(ii):$ \ A well known argument shows that, if $u$ is a nontrivial solution to \eqref{eq:problem+} and $J_q(u)<2c_q$, then $u$ does not change sign; see, for instance, \cite[Proof of Theorem A]{bc}. Assume that $u\geq 0$. By standard elliptic regularity (see \cite[Lemma 3.3]{CHS24}), $u\in W^{2,r}_{loc}(\rn)\cap \cC^{1,\alpha}_{loc}(\rn)$ for all $r\in[1,\infty)$ and $\alpha\in(0,1)$. Then, the strong maximum principle for strong solutions \cite[Theorem 9.6]{gt} applied to the equation $-\Delta u + {\mathds 1}_{\mathbb R^n\setminus\Omega} u^q = {\mathds 1}_\Omega u^q$
yields that $u>0$ in $\rn$.
\end{proof}
\smallskip

We are ready to prove Theorems~\ref{main1} and \ref{main11}.

\begin{proof}[Proof of Theorem~\ref{main1}]
Note that $\lim_{\xi\to\partial\Omega}\psi_\Omega(\xi)=\infty.$ Therefore, the function $\Psi_\o$, defined in \eqref{PSI}, has a global minimum point $\xi_0$, and $(d_0,\xi_0)$ is a global minimum of $\Psi$ for some $d_0>0$. By Proposition~\ref{red-en} and Lemma~\ref{iff:lem}, the reduced energy $\widetilde J_\eps$ has a critical point which gives rise to a solution $u_\eps$ of problem \eqref{eq:p} for $\eps$ small enough. Furthermore, $J_\eps(u_\eps)\to\mathfrak{a}:=\frac{1}{n}S^\frac{n}{2}$ as $\eps\to 0$. So, by Lemma~\ref{lem:positive solution}, $u_\eps$ is positive in $\rn$ for $\eps$ small enough.
\end{proof}
\smallskip

\begin{proof}[Proof of Theorem~\ref{main11}]
If $\xi_0$ is a nondegenerate critical point of $\psi_\Omega$, then $\Psi$ has a nondegenerate critical point. As nondegenerate critical points are stable under $\cC^1$-perturbations, the reduced energy $\widetilde J_\eps$ has a critical point and, arguing as in the proof of Theorem~\ref{main1}, we obtain the result.
\end{proof}
\smallskip

\begin{remark}
\emph{One could also consider the existence of solutions in the \emph{supercritical} regime, namely, when $\eps<0$ in \eqref{eq:p}. In this case, the main term of the reduced energy (see \eqref{em1} and \eqref{PSI}) would be
$\Psi(d,\xi)=c_1d ^n\psi_\Omega(\psi)+c_2\ln d$
which does not have any critical points in $d$. This suggests that the supercritical case does not have solutions with a single-bubble profile. On the other hand, inspired by \cite{dpfm}, it  would be interesting to study the existence of solutions in a domain with a hole using two (or more) bubbles as profiles. We leave this as an open problem.}
\end{remark}

\begin{remark} \label{rate:rem}\emph{(On how to find the Ansatz for the concentration rate)
Following the estimates in Proposition \ref{red-en} without using the Ansatz \eqref{para}, we would obtain that the leading terms for the expansion of $J(U_{\delta,\xi})$ are
\begin{align}\label{lt}
 H(\delta):=c_1-\eps(c_2+c_3\ln \delta)+c_4 \delta^n \psi_\Omega(\xi)
\end{align}
for some constants $c_1,c_2,c_3,c_4$ with $c_3,c_4>0.$ Recall that the concentration rate $\delta$ depends on $\eps$ and $\delta(\eps)\to 0$ as $\eps\to 0$ (as a consequence of the nonexistence of solutions to the limit problem). To implement the Lyapunov-Schmidt method, we search for suitable parameters that can yield critical points of \eqref{lt} that are bounded away from zero and are stable under uniform perturbations. This justifies the Ansatz $\delta = d \eps^\frac{1}{n}$ which gives the function
\begin{align*}
 h(d):=H(d \eps^\frac{1}{n})
 =c_1-\eps
 \left(
 c_2+c_3(\ln d + \frac{1}{n}\ln\eps)
 \right)+c_4 d^\frac{1}{n}\eps \psi_\Omega(\xi)
 =\eps \left(c_4 d^\frac{1}{n}\psi_\Omega(\xi)-c_3\ln d
 \right)+r(\eps)
\end{align*}
for some function $r$ that depends only on $\eps$. The Ansatz managed to join the summand $-\eps(c_2+c_3\ln \delta)$ and $c_4 \delta^n \psi_\Omega(\xi)$ to produce the function $c_4 d^\frac{1}{n}\psi_\Omega(\xi)-c_3\ln d$ that has a global minimum (in $d$).
}
\end{remark}

\section{The slightly subcritical problem: existence of nodal solutions}\label{sch:sec}

In this section we prove Theorem~\ref{main2}. The proof follows similar steps as in the previous section (see also \cite{ackclpi,bamipi} where similar estimates to ones we present here can be found).

As before, we rewrite \eqref{eq:p} as
\begin{equation}\label{p2n}
  u-i^* \left[Q_\Omega  f_\eps (u)\right]=0,\qquad   u\in E,\qquad f_\eps(u):=|u|^{{p-1}-\eps}u.
\end{equation}
We build a solution \eqref{p2} with the Ansatz $U_{\delta_1,\xi_1}-U_{\delta_1,\xi_2}+\phi,$ namely, two bubbles with opposite signs plus a small error. Considering the leading terms in the expansion of the reduced energy (see Proposition~\ref{red-en-n}), we make the following choices for the variables involved in the Ansatz. We assume that the concentration parameters $\delta_1$ and $\delta_2$ satisfy that
 \begin{equation}\label{dn}
 \delta_j
:=d_j\delta,\ \hbox{with}\ \delta:=\eps^\frac 1{n-2},\   d_j\in(0,\infty), \hbox{ and }j=1,2,  \end{equation}
and the concentration points are such that
\begin{equation}\label{xin}
\begin{aligned}
& \xi_j=\eta_j+\tau_j\nu_j\in\Omega,\ \hbox{where}\ \eta_j\in\partial\Omega,\ \nu_j\ \hbox{is the inner normal at $\eta_j$,}\\
&
 \tau_j=t_j\tau,\ \hbox{with}\ \tau:=\eps^\frac{2}{(n-2)(n+1)}\ \hbox{and}\   t_j\in(0,\infty).
 \end{aligned}
 \end{equation}
This implies that the concentration points $\xi_1$ and $\xi_2$ are inside $\Omega$ and are close to the boundary. Moreover, we assume that the concentration rate of each bubble is faster than the rate at which the concentration points approach the boundary of $\Omega,$ i.e.
\begin{equation}\label{dt}
{\frac{\delta}{\tau}}=\eps^{\frac{n-1}{(n-2)(n+1)}}\to 0\qquad \text{ as }\eps\to 0.
\end{equation}

Set
\begin{equation*}
K:=K_{\delta_1,\xi_1}\oplus K_{\delta_2,\xi_2},
\end{equation*}
where $K_{\delta,\xi}$ is given in \eqref{k}.

Recall the definition of $Z_{\delta_j,\xi_j}^{i}$ given in \eqref{z:def}.  Consider the orthogonal complement of $K$ in $D^{1,2}(\rn)$, namely,
\begin{equation*}
K^\perp:=\left\{u\in D^{1,2}(\mathbb R^n):\langle Z_{\delta_j,\xi_j}^{i},u\rangle=0,\ i=0,1,\dots,n;\ j=1,2\right\},
\end{equation*}
and the orthogonal projections in $D^{1,2}(\rn)$,
\begin{align*}
\Pi:D^{1,2}(\mathbb R^n)\to K\quad\text{ and }\quad
\Pi^\perp:D^{1,2}(\mathbb R^n)\to K^\perp.
\end{align*}
Arguing as in Lemma~\ref{lem:projections}, we also have that
\begin{align}\label{bdproy}
\|\Pi u\|_E\leq C \|u\|_E\quad\text{ and }\quad \|\Pi^\perp u\|_E\leq C \|u\|_E\quad\text{ for all }u\in E
\end{align}
and for some uniform constant $C>0$ with respect to $\eps.$

 We observe that $K$ and $\Pi$ depend on the free parameters $(d_1,d_2,t_1,t_2,\eta_1,\eta_2)\in (0,\infty)^4\times D_2\Omega$,
 where 
 $$D_2\Omega:=\left(\partial\Omega\times\partial\Omega\right) \setminus\{\xi_1=\xi_2\},$$
however we omit the dependencies for the sake of simplicity.
We also set
\begin{align*}
U_j:=U_{\delta_j,\xi_j},\ j=1,2.
\end{align*}
As in the previous section, \eqref{p2n} is equivalent to the system
\begin{equation}\label{sn}
\left\{\begin{aligned}\Pi ^\perp\left\{   U_1- U_2+\phi-i^* \left[Q_\Omega  f_\eps \left(  U_1- U_2+\phi\right)\right]\right\}&=0,\\
 \Pi \left\{   U_1- U_2+\phi-i^* \left[Q_\Omega  f_\eps \left(  U_1- U_2+\phi\right)\right]\right\}&=0.
\end{aligned}\right.
\end{equation}

We begin with an auxiliary lemma.  Recall that $\Omega$ is a smooth bounded subset of $\rn$.

\begin{lemma}\label{half:lemma}
 Let $\alpha>0$ and $\beta\geq 1$ so that $\alpha\beta>\frac{n}{n-2}$, then there is $C>0$ independent of $\eps>0$ so that
 \begin{align*}
|{\mathds 1}_{\mathbb R^n\setminus\Omega}|U_1-U_2|^\alpha|_\beta
  \leq C \delta^{\frac{n-2}{2}\alpha}\tau^{\frac{n}{\beta}-\alpha(n-2)}
 \end{align*}
 with $\delta$ and $\tau$ as in \eqref{dn} and \eqref{xin}.
\end{lemma}
\begin{proof}
Using a change of variables, for $i=1,2$,
\begin{align*}
\int_{\rn\backslash \Omega}|U_i|^{\alpha\beta }
&=\int_{\rn\backslash \Omega}\left|
\frac{\alpha_n \delta_i^{{\frac{n-2}{2}}}}{(\delta_i^2+|x-\eta_i-\tau_i\nu_i|^2)^\frac{n-2}{2}}
\right|^{\alpha\beta }\, dx\\
&=C\delta_i^{{\frac{n-2}{2}}\alpha\beta}\tau_i^{n-\alpha\beta(n-2)} \int_{\rn\backslash (\frac{\Omega-\eta_i}{\tau_i})}\left|
\frac{1}{((\frac{\delta_i}{\tau_i})^2+|y-\nu_i|^2)^\frac{n-2}{2}}
\right|^{\alpha\beta }\, dy\\
&=C\delta_i^{{\frac{n-2}{2}}\alpha\beta}\tau_i^{n-\alpha\beta(n-2)} \left(\int_{\rn\backslash H_i}
\frac{1}{|y-\nu_i|^{(n-2)\alpha\beta }}\, dy+o(1)\right),
\end{align*}
where we used \eqref{dt}, that $\frac{\Omega-\eta_i}{\tau_i}$ goes to $H_i:=\{y\in\mathbb R^n\ :\ \langle y,\nu_i\rangle \ge0\}$ as $\eps\to 0$ (recall that $\Omega$ is smooth), and that $\alpha\beta>\frac{n}{n-2}$. The claim now follows, since
$ |{\mathds 1}_{\mathbb R^n\setminus\Omega}|U_1-U_2|^\alpha|_\beta
 \leq C(|{\mathds 1}_{\mathbb R^n\setminus\Omega}|U_1|^\alpha|_\beta+|{\mathds 1}_{\mathbb R^n\setminus\Omega}|U_2|^\alpha|_\beta).$

\end{proof}

Now, we solve the first equation in \eqref{sn}.

\begin{proposition}\label{rid1n}
For any bounded open subset $\Gamma$ of $(0,\infty)^4\times D_2(\Omega)$ with $\overline{\Gamma}\subset(0,\infty)^4\times D_2(\Omega)$, there exists $\eps_0>0$ such that, for any $\eps\in(0,\eps_0)$ and for any
 $(d_1,d_2,t_1,t_2,\eta_1,\eta_2)\in\Gamma,$
there exists a unique $\phi=\phi(d_1,d_2,t_1,t_2,\eta_1,\eta_2,\eps)\in K^\perp $ which solves the first equation in \eqref{sn} and there is $C>0$ such that
\begin{align*}
\|\phi\|_E&\leq C  \eps^{{\frac{n+2}{2}}{\frac{n-1}{(n-2)(n+1)}}},\quad  \hbox{if}\ n\ge4,\\
\|\phi\|_E&\leq C\eps|\ln \eps|,\hspace{1.5cm} \hbox{if}\ n=3,
\end{align*}
for any $\eps\in(0,\eps_0)$. Moreover,  $\Gamma\ni (d_1,d_2,t_1,t_2,\eta_1,\eta_2)\to\phi(d_1,d_2,t_1,t_2,\eta_1,\eta_2)$ is a $\cC^1$-map.
\end{proposition}

\begin{proof}
Using that $U_j=i^* \left[ f_0 \left(U_j\right)\right]$ and that $h=Q_\Omega h+2{\mathds 1}_{\mathbb R^n\setminus\Omega} h$ for any $h:\rn\to \r,$ we have that the first equation in \eqref{s} can be rewritten as
\begin{align*}
&\underbrace{\Pi ^\perp\left\{  \phi-i^* \left[   f'_0 \left( U_1- U_2 \right)\phi \right]\right\}}_{=:\mathscr L(\phi)}\\
 &=\underbrace{\Pi ^\perp\circ i^*\left\{  Q_\Omega\left[  f_\eps \left( U_1- U_2+\phi \right) -
 f_\eps \left( U_1- U_2  \right) - f'_\eps \left( U_1- U_2\right)\phi  \right]\right\}}_{=:\mathscr N(\phi)}\\
&+\underbrace{\Pi ^\perp\circ i^*\left\{  Q_\Omega\left[  f'_\eps \left( U_1- U_2 \right)
 - f'_0\left( U_1- U_2\right)  \right]\phi\right\}
 +2\Pi^\perp\circ i^*\left\{  -{\mathds 1}_{\mathbb R^n\setminus\Omega}  f'_0\left( U_1- U_2\right) \phi\right\}}_{=:\mathscr L_\eps (\phi)}\\
&+\underbrace{\Pi ^\perp\circ i^*\left\{  
Q_\Omega\left[  f_\eps \left( U_1- U_2  \right)
 - f_0\left( U_1- U_2\right)  \right]\right\}}_{=:\mathscr E_1}\\ &+\underbrace{\Pi ^\perp\circ i^*\left\{  
f_0\left( U_1- U_2\right)-  f_0\left(U_1\right)+f_0\left(U_2\right)\right\}}_{=:\mathscr E_2}\\ &+2\underbrace{\Pi ^\perp\circ i^*\left\{
 -{\mathds 1}_{\mathbb R^n\setminus\Omega}  f_0\left(U_1-U_2\right)\right\}}_{=:\mathscr E_3}
\end{align*}

The linear operator $\mathscr L:K^\perp_{\delta,\xi}\cap L^s(\rn)\to K^\perp_{\delta,\xi}\cap L^s(\rn)$ given by
\begin{align*}
\mathscr L(\phi) := \Pi ^\perp\left\{  \phi-i^* \left[   f'_0 \left( U_1- U_2 \right)\phi \right]\right\}
\end{align*}
is invertible; this can be seen by adapting the proof of \cite[Lemma 1.7]{mupi} to this setting.

In the following, $C>0$ denotes possibly different constants independent of $\eps$.  The following estimates are standard, they  mostly rely on Lemma~\ref{lem:li} and can be found, for instance, in \cite[Lemma A.2]{bamipi} and \cite[Lemma A.3]{mupi}.

Note that the linear operator ${\mathscr L_\eps}:K^\perp_{\delta,\xi}\cap L^s(\rn)\to K^\perp_{\delta,\xi}\cap L^s(\rn)$ given by
\begin{align*}
 {\mathscr L_\eps}(\phi):=\Pi ^\perp\circ i^*\left\{  Q_\Omega\left[  f'_\eps \left( U_1- U_2 \right)
 - f'_0\left( U_1- U_2\right)  \right]\phi\right\}
 +2\Pi_{\delta,\xi}^\perp\circ i^*\left\{  -{\mathds 1}_{\mathbb R^n\setminus\Omega}  f'_0\left( U_1- U_2\right) \phi\right\}
\end{align*}
satisfies that
 \begin{align*}
\|{\mathscr L_\eps}(\phi)\|_E
&\leq C\left(
| f'_\eps \left(U_1-U_2  \right)- f'_0\left(U_1-U_2\right)|_{\frac{n}{2}}
+|{\mathds 1}_{\mathbb R^n\setminus\Omega} f'_0\left(U_1-U_2\right)|_{\frac{n}{2}}
\right) |\phi|_{\frac{2n}{n-2}},
 \end{align*}
 where we used \eqref{bdproy} and Lemma~\ref{lem:holder}.  Furthermore,
 \begin{align*}
| f'_\eps \left(U_1-U_2  \right)
 - f'_0\left(U_1-U_2\right)|_{\frac{n}{2}}\leq C\eps |\ln \eps|
 \end{align*}
 (cf. \eqref{re4} and \cite[Lemma B.13]{PST23}). By Lemma~\ref{half:lemma},
 \begin{align*}
 |{\mathds 1}_{\mathbb R^n\setminus\Omega} f'_0\left(U_1-U_2\right)|_{\frac{n}{2}}\leq C\left(\frac{\delta}{\tau}\right)^2
 \end{align*}
(cf. \eqref{re2}).  Therefore,
\begin{align*}
\|{\mathscr L_\eps}(\phi)\|_E \leq C \left(\eps|\ln \eps|+\left(\frac{\delta}{\tau}\right)^2\right)\|\phi\|_E.
 \end{align*}

 The nonlinear term $\mathscr N(\phi)$ satisfies that
 \begin{align}
  \|\mathscr N(\phi)\|\leq C (\|\phi\|^2_E+\|\phi\|^{p-\eps}_E)
 \end{align}
 (cf. \eqref{N:ref}). To estimate the error $\mathscr E:=\mathscr E_1+\mathscr E_2+\mathscr E_3$, where
 \begin{align*}
  \mathscr E_1&:=\Pi ^\perp\circ i^*\left\{
Q_\Omega\left[  f_\eps \left( U_1- U_2  \right)
 - f_0\left( U_1- U_2\right)  \right]\right\}, \\
  \mathscr E_2&:=\Pi ^\perp\circ i^*\left\{
f_0\left( U_1- U_2\right)-  f_0\left(U_1\right)+f_0\left(U_2\right)\right\}, \\
  \mathscr E_3&:=\Pi ^\perp\circ i^*\left\{
 -{\mathds 1}_{\mathbb R^n\setminus\Omega}  f_0\left(U_1-U_2\right)\right\},
 \end{align*}
 observe that, by Lemma~\ref{lem:holder} and arguing as in Lemma~\ref{lem:estimates},
 \begin{align*}
  \|\mathscr E_1\|_E\leq C\max\{|\mathscr E_1|_{\frac{2n}{n+2}},|\mathscr E_1|_{\frac{ns}{n+2s}}\}\leq C\eps|\ln \eps|
 \end{align*}
 (cf. \eqref{eq:iii} and \eqref{lol}).  Similarly,
 \begin{align*}
  \|\mathscr E_2\|_E\leq C\max\{|\mathscr E_2|_{\frac{2n}{n+2}},|\mathscr E_2|_{\frac{ns}{n+2s}}\},
 \end{align*}
where, by Lemma~\ref{Lemma:I_1aux},
 \begin{align*}
|f_0 \left(U_1-U_2   \right)-f_0\left(U_1\right)+f_0\left(U_2 \right)|_{\frac{2n}{n+2}}&\leq C
\left\{\begin{aligned}
&\delta^{\frac{n+2}{2}}\quad\quad \textit{if}\ n\ge7,\\
&\delta^4|\ln \delta|\quad  \textit{if}\ n=6,\\
&\delta^{n-2}\quad\quad \textit{if}\ n=3,4,5,
       \end{aligned}\right.\\
|f_0 \left(U_1-U_2   \right)-f_0\left(U_1\right)+f_0\left(U_2 \right)|_{\frac{2s}{n+2s}}&\leq C
\left\{\begin{aligned}
&\delta^{\frac{n+2}{2}}\quad\quad \textit{if}\ n\ge 6,\\
&\delta^{n-2}\quad\quad \textit{if}\ n=3,4,5.
       \end{aligned}\right.
 \end{align*}
Finally, by Lemma~\ref{Lemma:I_1aux}, \eqref{dn}, \eqref{xin}, and an easy calculation,
\begin{align*}
 \frac{|{\mathds 1}_{\mathbb R^n\setminus\Omega}  f_0\left(U_1-U_2 \right)|_{\frac{ns}{n+2s}}}{|{\mathds 1}_{\mathbb R^n\setminus\Omega}  f_0\left(U_1-U_2 \right)|_{\frac{2n}{n+2}}}
 \leq C \eps^{\frac{((n-1) n+2) s+4 n}{2 (n-2) (n+1) s}-\frac{n^2+n-2}{2 (n-2)
   (n+1)}}
   =C\eps^{\frac{2 s-n (s-2)}{(n-2) (n+1) s}}
 \to 0\qquad \text{ as }\eps\to0,
\end{align*}
because $s<\frac{2n}{n-2}$. Then, by Lemma~\ref{Lemma:I_1aux},
 \begin{align*}
  \|\mathscr E_3\|_E
  &\leq C\max\{|{\mathds 1}_{\mathbb R^n\setminus\Omega}  f_0\left(U_1-U_2 \right)|_{\frac{2n}{n+2}},|{\mathds 1}_{\mathbb R^n\setminus\Omega}  f_0\left(U_1-U_2 \right)|_{\frac{ns}{n+2s}}\}\\
  &=C|{\mathds 1}_{\mathbb R^n\setminus\Omega}  f_0\left(U_1-U_2 \right)|_{\frac{2n}{n+2}}
  \leq C\left(\frac{\delta}{\tau} \right)^{\frac{n+2}{2}}\leq C\eps^{{\frac{(n+2)(n-1)}{2(n-2)(n+1)}}}
\end{align*}
(cf. \eqref{eq:mp2}).  Then, using \eqref{dn}, \eqref{xin}, \eqref{dt}, and the estimates above, we find that
\begin{align*}
 \|\mathscr E\|_E&\leq C
 \left\{\begin{aligned}
         &\eps^{{\frac{(n+2)(n-1)}{2(n-2)(n+1)}}}\quad \hbox{ if }\ n\ge4,\\
         &\eps|\ln \eps|\hspace{1.1cm} \hbox{ if }\ n=3.
         \end{aligned}\right.
\end{align*}
Now we can conclude the proof as in Proposition~\ref{rid1}.
\end{proof}

To solve the second equation  in \eqref{sn}, we observe that $U_{\delta_1,\xi_1}-U_{\delta_2,\xi_2}+\phi$ solves \eqref{s} if and only if $$(d_1,d_2,t_1,t_2,\eta_1,\eta_2)\in (0,\infty)^4\times\partial\Omega^2$$
 is a critical point of the {\em reduced energy}
 \begin{align*}
\tilde J_\eps(d_1,d_2,t_1,t_2,\eta_1,\eta_2):=J_\eps\left(U_{\delta_1,\xi_1}-U_{\delta_2,\xi_2}+\phi\right)
 \end{align*}
 (cf. Lemma~\ref{iff:lem}), where
 $$J_\eps(u):=\frac12 \int_{\mathbb R^n}|\nabla u|^2-\frac 1{p+1-\eps}\int_{\mathbb R^n}Q_\Omega |u|^{p+1-\eps}.$$
 
 \begin{proposition}
\label{red-en-n}
The expansion
\begin{align}
\tilde J_\eps(d_1,d_2,t_1,t_2,\eta_1,\eta_2)&=\mathfrak a+\mathfrak b \eps+\mathfrak c \eps\ln\eps \notag\\
&+\eps\Xi(d_1,d_2,\eta_1,\eta_2)+\eps^{1+\frac2{(n-2)(n+1)}}\Upsilon(d_1,d_2,t_1,t_2,\eta_1,\eta_2)+o\left(\eps^{1+\frac2{(n-2)(n+1)}}\right),\label{zm1}
\end{align}
holds $\cC^0-$uniformly with respect to $(d_1,d_2,t_1,t_2,\eta_1,\eta_2)$ in compact sets of $ (0,\infty)^4\times D_2\Omega.$
Here,
\begin{align*}
\Xi(d_1,d_2,\eta_1,\eta_2):=\mathfrak c_1{ \frac{(d_1d_2)^{\frac{n-2}{2}}}{|\eta_1-\eta_2|^{n-2}}}-\mathfrak c_2\ln (d_1d_2)
\end{align*}
and
\begin{align*}
\Upsilon(d_1,d_2,t_1,t_2,\eta_1,\eta_2):=-\mathfrak c_3(d_1d_2)^{\frac{n-2}{2}}\langle \eta_1 -\eta_2,t_1\nu_1-t_2\nu_2\rangle
+\mathfrak c_4 \left[\left(\frac{d_1}{t_1}\right)^n+\left(\frac{d_2}{t_2}\right)^n\right].
\end{align*}
All the constants only depend on $n$ and the constants $\mathfrak c_1,\mathfrak c_2,\mathfrak c_3,$ and $\mathfrak c_4$ are positive.
\end{proposition}
 \begin{proof}
Recall that $U_j:=U_{\delta_j,\xi_j}.$ Arguing as in Proposition~\ref{red-en}, we have that
 \begin{align}\label{z0}
J_\eps\left(U_1-U_2+\phi\right)=J_\eps\left(U_1-U_2 \right)+O(\|\phi\|_E^{2}).
 \end{align}
 By Proposition~\ref{rid1n},
 \begin{align*}
\|\phi\|_E^{2}=o\left(\eps^{1+\frac2{(n-2)(n+1)}}\right).
 \end{align*}
Note that
\begin{align*}
 \int_{\mathbb R^n}|\nabla (U_1-U_2)|^2
 =\int_{\mathbb R^n}|\nabla U_1|^2+|\nabla U_1|^2-2\nabla U_1\nabla U_2
 =\int_{\mathbb R^n}|\nabla U_1|^2+|\nabla U_1|^2-2U_1^p U_2.
\end{align*}

 Hence,
 \begin{align}
 J_\eps\left(U_1-U_2 \right)&=\frac{1}{2}\int_{\mathbb R^n}|\nabla (U_1-U_2)|^2
 -\frac{1}{p+1-\eps}\int_{\Omega}|U_1-U_2|^{p+1-\eps}+
 \frac{1}{p+1-\eps}\int_{\mathbb R^n\setminus\Omega}| U_1-U_2|^{p+1-\eps}\notag\\
 &=\frac{1}{2}\int_{\mathbb R^n}|\nabla (U_1-U_2)|^2
 -\frac{1}{p+1-\eps}\int_{\mathbb R^n}| U_1-U_2|^{p+1-\eps}+
2 \frac{1}{p+1-\eps}\int_{\mathbb R^n\setminus\Omega}| U_1-U_2|^{p+1-\eps}\notag\\
 &=\sum_{j=1,2}\left(\frac12\int_{\mathbb R^n}|\nabla U_j|^2-\frac1{p+1}\int_{\mathbb R^n}| U_j|^{p+1}\right)\notag\\
 &-\int_{\mathbb R^n} U_1^pU_2+\frac1{p+1}\int_{\mathbb R^n}| U_1|^{p+1}+\frac1{p+1}\int_{\mathbb R^n}| U_2|^{p+1}-\frac1{p+1}\int_{\mathbb R^n}|U_1-U_2|^{p+1}\notag
\\
&+
2 \frac1{p+1}\int_{\mathbb R^n\setminus\Omega}| U_1-U_2|^{p+1 }\notag\\
&+\frac1{p+1}\int_{\mathbb R^n}|U_1-U_2|^{p+1} -\frac1{p+1-\eps}\int_{\mathbb R^n}| U_1-U_2|^{p+1-\eps}\notag\\
&+2 \frac1{p+1-\eps}\int_{\mathbb R^n\setminus\Omega}| U_1-U_2|^{p+1-\eps}-2 \frac1{p+1}\int_{\mathbb R^n\setminus\Omega}| U_1-U_2|^{p+1}\notag\\
&=\mathfrak A+I_1+I_2+I_3+I_4,\label{z1}
 \end{align}
where
\begin{align*}
 \mathfrak A&:=\sum_{j=1,2}\left(\frac12\int_{\mathbb R^n}|\nabla U_j|^2-\frac1{p+1}\int_{\mathbb R^n}| U_j|^{p+1}\right),\\
 I_1&:= -\int_{\mathbb R^n} U_1^pU_2+\frac1{p+1}\int_{\mathbb R^n}| U_1|^{p+1}+\frac1{p+1}\int_{\mathbb R^n}| U_2|^{p+1}-\frac1{p+1}\int_{\mathbb R^n}|U_1-U_2|^{p+1},\\
 I_2&:=2 \frac1{p+1}\int_{\mathbb R^n\setminus\Omega}| U_1-U_2|^{p+1 },\\
 I_3&:=\frac1{p+1}\int_{\mathbb R^n}|U_1-U_2|^{p+1} -\frac1{p+1-\eps}\int_{\mathbb R^n}| U_1-U_2|^{p+1-\eps},\\
 I_4&:=2 \frac1{p+1-\eps}\int_{\mathbb R^n\setminus\Omega}| U_1-U_2|^{p+1-\eps}-2 \frac1{p+1}\int_{\mathbb R^n\setminus\Omega}| U_1-U_2|^{p+1 }.
\end{align*}

By the proof of \cite[Lemma 3.2]{dpfm}, there are $\mathfrak B\in \r$ and $\mathfrak C>0$ such that
\begin{align}\label{z2}
 I_3=-\eps\left(\mathfrak B+\mathfrak C\ln (\delta_1\delta_2)+O(\eps)\right)\qquad \text{ as }\eps\to 0
\end{align}
(cf. \eqref{ee3}).  Moreover, by the mean value theorem and arguing as in Lemma~\ref{half:lemma}, it is not hard to show that
\begin{align*}
 I_4
 &=2 \frac1{p+1-\eps}\left(\int_{\mathbb R^n\setminus\Omega}| U_1-U_2|^{p+1-\eps}
  -\int_{\mathbb R^n\setminus\Omega}| U_1-U_2|^{p+1}\right)
 +2 \left(\frac1{p+1-\eps}-\frac1{p+1}\right)\int_{\mathbb R^n\setminus\Omega}| U_1-U_2|^{p+1}\\
& =o(\eps^2)\qquad \text{as $\eps\to 0.$}
\end{align*}

Next we estimate $I_1$ and $I_2.$ Let $r>0$ be such that $B_1\cap B_2=\emptyset,$ where $B_j:=B(\eta_j,r).$ Using Lemma~\ref{lem:li},
\begin{align*} \int_{\mathbb R^n}|U_1-U_2|^{p+1}&=\int_{B_1}|U_1-U_2|^{p+1}+
\int_{B_2}|U_1-U_2|^{p+1}+\int_{\mathbb R^n\setminus B_1\cup B_2}|U_1-U_2|^{p+1}\\
&=\int_{B_1}|U_1 |^{p+1}-(p+1)\int_{B_1}|U_1 |^{p }U_2
+\int_{B_2}|U_ 2|^{p+1}-(p+1)\int_{B_2}|U_2 |^{p }U_1+O(\delta^n)\\
&=\int_{\mathbb R^n}|U_1 |^{p+1}-(p+1)\int_{\mathbb R^n}|U_1 |^{p }U_2
+\int_{\mathbb R^n}|U_ 2|^{p+1}-(p+1)\int_{\mathbb R^n}|U_2 |^{p }U_1+O(\delta^n),\end{align*}
and so, there are $\mathfrak D,\mathfrak E>0$ such that
\begin{align}
I_1
&=\int_{\mathbb R^n}|U_2 |^{p }U_1+O(\delta^n)\notag\\
&=\int_{\mathbb R^n}\frac{\alpha_n^\frac{n+2}{n-2} \delta_1^\frac{n+2}{2}}{(\delta_1^2+|x-\xi_1|^2)^\frac{n+2}{2}}\frac{\alpha_n \delta_2^\frac{n-2}{2}}{(\delta_2^2+|x-\xi_2|^2)^\frac{n-2}{2}}\, dx+O(\delta^n)\notag\\
&=\alpha_n^{\frac{n+2}{n-2}+1} \delta_1^{\frac{n-2}{2}}\delta_2^\frac{n-2}{2}
\int_{\mathbb R^n}
\frac{1}{(1+|y|^2)^\frac{n+2}{2}}
\frac{1}{(\delta_2^2+|\delta_1 y+\xi_1-\xi_2|^2)^\frac{n-2}{2}}\, dy+O(\delta^n)\notag\\
&=(\mathfrak D+o(1))(\delta_1\delta_2)^{\frac{n-2}{2}}\frac{1}{|\xi_1-\xi_2|^{n-2}}+O(\delta^n)\notag\\
&=(\mathfrak D+o(1))(\delta_1\delta_2)^{\frac{n-2}{2}}\left(\frac{1}{|\eta_1-\eta_2|^{n-2}}-\mathfrak E \left\langle \frac{\eta_1 -\eta_2}{|\eta_1-\eta_2|^n},\tau_1\nu_1-\tau_2\nu_2\right\rangle\right)+O(\delta^{n-2}\tau^2)+O(\delta^n),\label{z3}
\end{align}
because $|\xi_1-\xi_2|^{ 2}=|\eta_1+\tau_1\nu_1-\eta_2-\tau_2\nu_2|^2=|\eta_1 -\eta_2|^2+2\langle \eta_1 -\eta_2,\tau_1\nu_1-\tau_2\nu_2\rangle+O(\tau^2)$ and hence, using the mean value theorem,
\begin{align*}
\frac{1}{|\xi_1-\xi_2|^{n-2}}=
\frac{1}{|\eta_1-\eta_2|^{n-2}}
-(n-2)\frac{\langle \eta_1 -\eta_2,\tau_1\nu_1-\tau_2\nu_2\rangle}{|\eta_1-\eta_2|^n}+O(\tau^2).
\end{align*}
Finally, arguing as in Lemma~\ref{half:lemma},
 \begin{align}
 I_2
 =2 \frac1{p+1}\int_{\mathbb R^n\setminus\Omega}| U_1-U_2|^{p+1 }
 =\mathfrak F\left[ \left(\frac{\delta _1}{\tau_1}\right)^n(1+o(1))+\left(\frac{\delta _2}{\tau_2}\right)^n(1+o(1))\right]\label{z4}
 \end{align}
 for some $\mathfrak F\in \R$ as $\eps\to 0$. The expansion \eqref{zm1} now follows by \eqref{z0}, \eqref{z1}, \eqref{z2}, \eqref{z3}, and \eqref{z4}.

 \end{proof}
 \smallskip

 \begin{proof}[Proof of Theorem~\ref{main2}: completed]
 It suffices to show that the reduced energy $\widetilde J_\eps$ has a critical point (cf. Lemma~\ref{iff:lem}).

Consider the change of variables $d_1=r$ and $d_2={\frac{s}{r}}$. Then, it suffices to find a global minimum of the function
\begin{align*} \Psi_\eps(s ,r,t_1,t_2,\eta_1,\eta_2)&=\mathfrak c_1{ \frac{s^{n-2}}{|\eta_1-\eta_2|^{n-2}}}-2\mathfrak c_2\ln s\\ &
+\eps^{ \frac2{(n-2)(n+1)}}  \left\{-\mathfrak c_3 s^{n-2 }\left\langle \frac{\eta_1 -\eta_2}{|\eta_1-\eta_2|^n},t_1\nu_1-t_2\nu_2\right\rangle
+\mathfrak c_4 \left[\left(\frac{r}{t_1}\right)^n+\left(\frac{s}{r t_2}\right)^n\right]\right\}\\ &
+o\left(\eps^{ \frac2{(n-2)(n+1)}}\right).
\end{align*}

Let $\bar \eta_1,\bar \eta_2\in\partial\Omega$ be such that
\begin{align}\label{etaseq}
r:=|\bar \eta_1-\bar \eta_2|=\max\{ |\eta_1-\eta_2|\ : \ \eta_1,\eta_2\in\partial\Omega\}
\end{align}
and let $\bar s>0$ be the minimum of the function
\begin{align*}
s\to\mathfrak c_1{ \frac{s^{n-2}}{|\bar \eta_1-\bar \eta_2|^{n-2}}}-2\mathfrak c_2\ln s,\qquad s>0.
\end{align*}

Let ${\mathcal U}:=\left\{(s,\eta_1,\eta_2)\in \left(\frac{\bar s}{2},2\bar s\right)\times \partial \Omega \times \partial \Omega\::\: |\eta_1-\eta_2|>\frac{r}{2}\right\}$
and let $\Psi_1:{\mathcal U}\to \r$ be given by
\begin{align*}
\Psi_1(s,\eta_1,\eta_2):=
 \mathfrak c_1{ \frac{s^{n-2}}{|\eta_1-\eta_2|^{n-2}}}-2\mathfrak c_2\ln s.
\end{align*}
It is easy to see that $(\bar s,\bar \eta_1,\bar \eta_2)\in \mathcal U$ is the minimum of $\Psi_1$ in $\mathcal U$. Now, consider the function $ \Psi_2:{\mathcal U}\times(0,\infty)^3\to \r$ given by
\begin{align*}
 \Psi_2(s,\eta_1,\eta_2,r,t_1,t_2):=\left\{-\mathfrak c_3 s^{n-2 }\left\langle \frac{\eta_1 -\eta_2}{|\eta_1-\eta_2|^n},t_1\nu_1-t_2\nu_2\right\rangle
+\mathfrak c_4 \left[\left(\frac{r}{t_1}\right)^n+\left(\frac{s}{r t_2}\right)^n\right]\right\}
\end{align*}

Since $\o$ is a bounded smooth domain in $\rn$ and \eqref{etaseq} holds, then the interior unit normal to $\partial\o$ at $\bar\eta_1$ is $\frac{\bar\eta_2-\bar\eta_1}{|\bar\eta_2-\bar\eta_1|}$. \ Indeed, the ball $B$ centered at $\bar\eta_2$ of radius $|\bar\eta_2-\bar\eta_1|$ contains $\o$. Therefore, the tangent space to $\partial\o$ at $\bar\eta_1$ coincides with the tangent space to the sphere $\partial B$ at $\bar\eta_1$, which is the space orthogonal to $\bar\eta_2-\bar\eta_1$. Therefore,
\begin{align*}
 \bar \lambda:=-\left\langle \bar \nu_1, \frac{\bar\eta_1- \bar\eta_2}{|\bar\eta_1- \bar\eta_2|^n}\right\rangle=-\left\langle  \bar\nu_2, \frac{\bar\eta_2- \bar\eta_1}{|\bar\eta_1- \bar\eta_2|^n}\right\rangle>0.
\end{align*}
Hence, the function
\begin{align*}
 (0,\infty)^3 \ni (r,t_1,t_2)\mapsto \Psi_2(\bar s,\bar \eta_1,\bar \eta_2,r,t_1,t_2):= \mathfrak c_3 \bar s^{n-2 }\bar \lambda (t_1+t_2)
+\mathfrak c_4 \left[\left(\frac{r}{t_1}\right)^n+\left( \frac{\bar s}{r t_2}\right)^n\right]
\end{align*}
has a strict minimum at some point $(\bar r,\bar t_1,\bar t_2)\in (0,\infty)^3.$ Consider the set
${\mathcal V}:=(\bar r/2,2\bar r)\times(\bar t_1/2,2\bar t_1)\times(\bar t_2/2,2\bar t_2).$

Extending trivially the function $\Psi_1$ to $\mathcal U\times \mathcal V$, it is not hard to see that the continuous function
\begin{align*}
\Psi_1+\eps^{ \frac2{(n-2)(n+1)}} \Psi_2:\overline{\mathcal U\times \mathcal V}\to \r
\end{align*}
achieves its minimum at a point $(s_\eps,\eta_{1,\eps},\eta_{1,\eps},r_\eps,t_{1,\eps},t_{2,\eps})\in \mathcal U\times \mathcal V$ such that
\begin{align*}
 (s_\eps,\eta_{1,\eps},\eta_{1,\eps},r_\eps,t_{1,\eps},t_{2,\eps})\to (\bar s,\bar \eta_1,\bar \eta_2,\bar r,\bar t_1,\bar t_2)\qquad \text{ as }\eps\to 0.
\end{align*}
Since this critical point is stable under $C_0$-perturbations (because it is the global minimum), this concludes the proof.
\end{proof}

 \section{The critical problem in a domain with a shrinking hole}
\label{sec:critical}

In this section we prove Theorem~\ref{main3}. We follow the same scheme as before, focusing only on the key points of the proof. We refer to Section~\ref{sec:positive solution} and to the papers \cite{gmp,mupi1,mupi2} for further details.

Recall that, by assumption, $0\in\o$. As before, we rewrite the problem \eqref{eq:p_critical} as
\begin{equation}\label{p2hole}
  u-i^* \left[Q_{\Omega_\vr} f_0 (u)\right]=0,\qquad u\in D^{1,2}(\mathbb R^n),
\end{equation}
where $f_0(u)=u^p$ and $\o_\vr:=\{x\in\o:|x|>\vr\}$.  We look for a solution with the Ansatz $U_{\delta,\xi}+\phi,$ where
 \begin{equation}\label{parahole}
 \delta:=d\sqrt\vr\ \hbox{ \ with \ }\ d\in(0,\infty), \qquad \xi=\delta\zeta\ \hbox{ \ with \ }\  \zeta\in\mathbb R^n, \qquad\text{and}\qquad\phi\in K_{\delta,\xi}^\perp,
 \end{equation}
with $K_{\delta,\xi}^\perp $ as in \eqref{ko}.  The  choice of $\delta$ is motivated by the leading terms in the  expansion of the energy given in proof of the Proposition~\ref{red-enhole}. Problem \eqref{p2hole} is equivalent to the system
\begin{equation}\label{shole}
\begin{cases}
\Pi_{\delta,\xi}^\perp\left(U_{\delta,\xi}+\phi-i^* \left[Q_{\Omega_\vr} f_0 \left(U_{\delta,\xi}+\phi\right)\right]\right)=0, \\
\Pi_{\delta,\xi} \left(U_{\delta,\xi}+\phi-i^* \left[Q_{\Omega_\vr} f_0 \left(U_{\delta,\xi}+\phi\right)\right]\right)=0.
\end{cases}
\end{equation}
The following result gives a solution for the first equation in \eqref{shole}.

\begin{proposition}\label{rid1hole}
Given an open bounded set $\t$ with $\overline{\t}\subset (0,\infty)\times\mathbb R^n$, there exist $c>0$ and $\vr_0>0$ such that, for each $\vr\in(0,\vr_0)$ and $(d,\zeta)\in \t$, there exists a unique $\phi_{\delta,\xi}\in K^\perp_{\delta,\xi}$ with $\delta=d\sqrt\vr$ and $\xi=\delta\zeta$ which satisfies
\begin{align*}
\|\phi_{\delta,\xi}\|\leq c\vr^\frac{n+6}4
\end{align*}
and solves the first equation in \eqref{shole}. Furthermore, for each $\vr\in(0,\vr_0)$, the map $(d,\zeta)\mapsto\phi_{\delta,\xi}$ is of class $\cC^1$ in $\overline{\t}$.
\end{proposition}

\begin{proof}
Noting that $U_{\delta,\xi}=i^* \left[ f_0 \left(U_{\delta,\xi}\right)\right]$, the first equation in \eqref{shole} can be rewritten as
\begin{align*}
\underbrace{\Pi_{\delta,\xi}^\perp\left[\phi-i^* \left[   f'_0 \left(U_{\delta,\xi} \right)\phi \right]\right]}_{=:\mathscr L(\phi)}
 &=\underbrace{(\Pi_{\delta,\xi}^\perp\circ i^*)\left(  Q_{\Omega_\vr}\left[  f_0 \left(U_{\delta,\xi}+\phi \right) -
 f_0 \left(U_{\delta,\xi}  \right) - f'_0\left(U_{\delta,\xi}\right)\phi  \right]\right)}_{=:\mathscr N(\phi)}\\
&\quad +\underbrace{ (\Pi_{\delta,\xi}^\perp\circ i^*)\left(  -{\mathds 1}_{\mathbb R^n\setminus{\Omega_\vr}}  f'_0\left(U_{\delta,\xi}\right) \phi\right)}_{=:\mathscr L_\vr(\phi)}+\underbrace{ (\Pi_{\delta,\xi}^\perp\circ i^*)\left(
 -{\mathds 1}_{\mathbb R^n\setminus{\Omega_\vr}}  f_0\left(U_{\delta,\xi}\right)\right)}_{=:\mathscr E}.
\end{align*}
The linear operator $\mathscr L:K^\perp_{\delta,\xi} \to K^\perp_{\delta,\xi}$ is invertible and there is a constant $C>0$ such that, for small enough $\vr$,
$$\|\mathscr L^{-1}(\psi)\|\leq C\|\psi\|\qquad\text{for every \ }\psi\in K^\perp_{\delta,\xi}, \ (d,\zeta)\in\overline{\t},$$
see Lemmas~\ref{B1} and~\ref{B2}.

In the following we use $C>0$ to denote possibly different constants independent of $\vr$. Setting $\eps=0$ in the proof of Lemma~\ref{lem:estimates}$(i)$ we see that the new nonlinear operator $\mathscr N$ defined above satisfies
\begin{align*}
\|\mathscr N(\phi)\|\leq C(\|\phi\|^2+\|\phi\|^p)\qquad\text{for every \ }\phi\in K_{\delta,\xi}^\perp.
\end{align*}
The linear operator ${\mathscr L_\vr}$ satisfies
\begin{align*}
\|{\mathscr L_\vr}(\phi)\|
\leq C\left(|\mathds 1_{\rn\setminus\o}f'_0(U_{\delta,\xi})|_\frac n2 + |\mathds 1_{B_\vr}f'_0(U_{\delta,\xi})|_\frac n2\right) |\phi|_ \frac{2n}{n-2} \leq C\Big(\delta^2+\left(\frac{\rho}{\delta}\right)^2\Big)\|\phi\|\qquad\text{for every \ }\phi\in K_{\delta,\xi}^\perp,
\end{align*}
and the error $\mathscr E$ satisfies the estimate
\begin{align*}
\|\mathscr E\|\leq C\left(|\mathds 1_{\rn\setminus\o}f_0(U_{\delta,\xi})|_\frac{2n}{n+2} + |\mathds 1_{B_\vr}f_0(U_{\delta,\xi})|_\frac{2n}{n+2}\right)
 \leq C \Big(\delta^{\frac{n+2}{2}}+\delta^2\left(\frac{\vr}{\delta}\right)^\frac{n+2}2\Big)
 \leq C \vr^{\frac{n+6}{4}},
\end{align*}
where $B_\vr$ is the ball of radius $\vr$ centered at $0$. The rest of the proof is like that of Proposition~\ref{rid1}.
\end{proof}

To solve the second equation in \eqref{shole}, arguing as in Lemma~\ref{iff:lem}, we prove first that $U_{\delta,\xi}+\phi_{\delta,\xi}$ solves \eqref{shole} if and only if $(d,\zeta)\in (0,\infty)\times\mathbb R^n$ is a critical point of the reduced energy
$$\tilde J_\vr(d,\zeta):=J_\vr\left(U_{\delta,\xi}+\phi_{\delta,\xi}\right)$$
 where $\delta=d\sqrt\vr$, $\xi=\delta\zeta$, and
 $$J_\vr(u):=\frac12 \int_{\mathbb R^n}|\nabla u|^2-\frac 1{p+1}\int_{\mathbb R^n}Q_{\Omega_\vr}|u|^{p+1}.$$
Then, we derive the following expansion for $\tilde J_\vr$.

\begin{proposition}
\label{red-enhole}
Let $\t$ be an open bounded set with $\overline{\t}\subset (0,\infty)\times\mathbb R^n$. Then, the expansion
$$\tilde J_\vr(d,\zeta)=\mathfrak a+\vr^\frac n2\Phi(d,\zeta)+o\left(\vr^\frac n2\right),$$
holds true $\cC^1$-uniformly in $\overline{\t}$ as $\vr\to 0$, where
$$\Phi(d,\zeta):=\mathfrak b_1 d^n+\mathfrak b_2\frac1{d^n}\frac1{(1+|\zeta|^2)^n}.$$
All the constants only depend on $n$, \ $\mathfrak{a}=\frac{1}{n}S^\frac{n}{2}$, \ and the constants $\mathfrak b_1$ and $\mathfrak b_2$ are positive.
\end{proposition}

\begin{proof}
Recall that $\|\phi_{\delta,\xi}\|<c\vr^\frac{n+6}{4}$ (by Proposition~\ref{rid1hole}). Then, replacing $Q_\o$ with $Q_{\o_\vr}$ and setting $\eps=0$ in \eqref{eq:J1} and \eqref{eq:J4}, and using \eqref{eq:J2}, we see that
 $$J_\vr(U_{\delta,\xi}+\phi_{\delta,\xi}) = J_\vr(U_{\delta,\xi})+ O(\vr^{\frac{n+6}{2}}).$$
Furthermore,
\begin{align*}
J_\vr\left(U_{\delta,\xi} \right)&=\frac12\irn|\nabla U_{\delta,\xi}|^2 -\frac1{p+1 }\int_{\Omega_\vr}| U_{\delta,\xi}|^{p+1}+
 \frac1{p+1 }\int_{\mathbb R^n\setminus\Omega_\vr}| U_{\delta,\xi}|^{p+1 }\\
 &=\frac12\int_{\mathbb R^n}|\nabla U_{\delta,\xi}|^2
 -\frac1{p+1 }\int_{\mathbb R^n}| U_{\delta,\xi}|^{p+1 }+
\frac{2}{p+1}\int_{\mathbb R^n\setminus\Omega_\vr}| U_{\delta,\xi}|^{p+1}\\
 &=\underbrace{\frac12\int_{\mathbb R^n}|\nabla U_{\delta,\xi}|^2-\frac1{p+1}\int_{\mathbb R^n}| U_{\delta,\xi}|^{p+1}}_{=:\mathfrak a} + \frac{2}{p+1}\int_{\mathbb R^n\setminus\Omega}| U_{\delta,\xi}|^{p+1 } -\frac{2}{p+1}\int_{ B_\vr}| U_{\delta,\xi}|^{p+1 },
 \end{align*}
where $B_\vr$ is the ball of radius $\vr$ centered at $0$. A straightforward computation shows that
\begin{align*}
\frac{2}{p+1}\int_{\mathbb R^n\setminus\Omega}| U_{\delta,\xi}|^{p+1 }&=\delta^n\Big(\frac{2}{p+1}\alpha_n^{p+1}\int_{\mathbb R^n\setminus\Omega}\frac 1{|x-\xi|^{2n}}+o(1)\Big)\\
&=\vr^\frac{n}{2}d^n\Big(\underbrace{\frac{2}{p+1}\alpha_n^{p+1}\int_{\mathbb R^n\setminus\Omega}\frac 1{|x|^{2n}}}_{=:\mathfrak b_1}\Big)+o(\vr^\frac{n}{2}).
\end{align*}
Now, setting $x=\delta y$ yields
\begin{align*}
&\frac{2}{p+1}\int_{ B_\vr}| U_{\delta,\xi}|^{p+1 }=\frac2{p+1 }\alpha_n^{p+1}\int_{|x|\le\vr}{\frac{\delta^n}{(\delta^2+|x-\delta\zeta|^2)^n}}\d x=\frac2{p+1 }\alpha_n^{p+1}\int_{|y|\le\rho/\delta}{\frac{1}{(1+|y-\zeta|^2)^n}}\d y\\
&\qquad =\frac{2}{p+1}\alpha_n^{p+1}|B_1|\left(\frac{\vr}{\delta}\right)^n\left[ {\frac{1}{(1+|\zeta|^2)^n}}+o(1)\right] =\underbrace{\frac{2}{p+1 }\alpha_n^{p+1}|B_1|}_{=:\mathfrak b_2}\vr^\frac{n}{2}\frac{1}{d^n}\left[ {\frac{1}{(1+|\zeta|^2)^n}}\right]+o(\vr^\frac{n}{2}).
\end{align*}
A standard argument shows that the expansion holds true $\cC^1$-uniformly in $\overline{\t}$ as $\vr\to 0$; see, e.g., \cite[Section 7]{gmp}.
\end{proof}
\smallskip

\begin{proof}[Proof of Theorem~\ref{main3}(completed)]
The point $(d_0,0)$ with $d_0=(\frac{\mathfrak{b}_2}{\mathfrak{b}_1})^\frac{1}{2n}$ is the only critical point of the function $\Phi$. It is a nondegenerate saddle point, stable under $\cC^1$-perturbations. Therefore, for small enough $\vr$, the reduced energy $\widetilde J_\vr$ has a critical point that converges to $(d_0,0)$ as $\vr\to 0$, and problem \eqref{eq:p_critical} has a solution $u_\vr$ concentrating at $0$, whose energy satisfies $J_\vr(u_\vr)\to \mathfrak{a}=\frac{1}{n}S^\frac{n}{2}$. By Lemma~\ref{lem:positive solution}$(ii)$, we have that $u_\vr$ is strictly positive in $\rn$.
 \end{proof}

 \appendix

\section{A generic property}

 Let $\Omega\subset \rn$ be a smooth open bounded set. Here we prove that, up to a small perturbation, the function
 \begin{align*}
\psi_\Omega(x):= \int_{\mathbb R^{n}\setminus \Omega}\frac1{|x-y|^{2n}}dy,\qquad x\in\Omega,
 \end{align*}
is a Morse function. Recall that a $\cC^2$-function $\theta:\o\to\r$ is a \emph{Morse function} if all of its critical points are nondegenerate. To be more precise, for $m\in\mathbb N,$ let $\mathfrak{C}^m $ denote the Banach space of all the $\cC^m$-maps $\theta:\mathbb R^n\to\mathbb R^n$ such that
\begin{equation*}
\|\theta\|_m:=\max_{i=1,\ldots,n}\max\limits_{0\le|\ell|\le m}\sup\limits_{x\in\mathbb R^n}\left|
\frac{\partial ^\ell \theta_i(x)}{\partial_ {x_1}^{\ell_1}\dots\partial _{x_n}^{\ell_n}}
\right|<\infty,
\end{equation*}
equipped with the norm $\|\cdot\|_m$.
Let $\mathfrak{B}_\rho:=\left\{\theta\in \mathfrak{C}^2 \ :\ \|\theta\|_2\le\rho\right\}$
be the ball of radius $\rho$ in $\mathfrak{C}^2$ centered at $0$, and set
$$\Omega_\theta:=\{x+\theta x:x\in\o\},\qquad \theta\in\mathfrak{C}^2.$$
We prove the following result.

\begin{theorem}\label{mainap}
There exists $\rho>0$ such that the set
$$\{\theta\in \mathfrak{B}_\rho: \psi_{\o_\theta}\text{ is a Morse function}\}$$
is a residual\textemdash hence, dense\textemdash subset of $\mathfrak{B}_\rho$.
\end{theorem}

Recall that \emph{$A$ is a residual subset of $B$} if $B\setminus A$ is a countable union of closed subsets with empty interior.

Consider the function $f:\o\times\mathfrak{B}_\rho\to\r$ given by
$$f(x,\theta):=\psi_{\o_\theta}(x+\theta(x)),$$
and define $F:\o\times\mathfrak{B}_\rho\to\rn$ to be
\begin{equation}\label{eq:F}
F(x,\theta):=\nabla_xf(x,\theta)=\nabla\psi_{\o_\theta}(x+\theta(x))[I+D\theta(x)].
\end{equation}
We begin with the following remark. 

\begin{lemma}
There exists $\rho_0>0$ such that, if $\theta\in\mathfrak{B}_\rho$, $\rho\in(0,\rho_0)$, and $0$ is a regular value of the function $x\mapsto F(x,\theta)$, then $\psi_{\o_\theta}$ is a Morse function.
\end{lemma}

\begin{proof}
By definition, $0$ is a regular value of $x\mapsto F(x,\theta)$ if and only if $D_xF(x,\theta)$ is an isomorphism for every $x$ such that $F(x,\theta)=0$. Now, if $\rho$ is small enough, then $I+D\theta(x)$ is a linear isomorphism. Therefore, $F(x,\theta)=0$ if and only if $\nabla\psi_{\o_\theta}(x+\theta(x))=0$, i.e., if and only if $x+\theta(x)$ is a critical point of $\psi_{\o_\theta}$. Similarly, if $D_xF(x,\theta)$ is an isomorphism, then $x+\theta(x)$ is nondegenerate.
\end{proof}
 
So, to prove Theorem~\ref{mainap}, we will show that, for $\rho$ small enough, the set
$$\{\theta\in \mathfrak{B}_\rho:0\text{ is a regular value of the function }x\mapsto F(x,\theta)\}$$
is a residual subset of $\mathfrak{B}_\rho$. We use the following abstract transversality theorem, see \cite{Q,ST,U}.

 \begin{theorem}\label{tran}
 Let $X,Y,Z$ be Banach spaces and $U\subset X,$ $V\subset Y$ be open subsets.
 Let $F:U\times V\to Z$ be a $C^k$-map with $k\ge1.$ Assume that
 \begin{itemize}
 \item[$i)$] for any $y\in V$, \ $F(\cdot,y):U\to Z$ is a Fredholm map of index $m$ with $m\le k,$
 \item[$ii)$] $0$ is a regular value of $F$, i.e., the linear operator $F'(x,y):X\times Y\to Z$ is onto at any point $(x,y)$ such that $F(x,y)=0,$
 \item[$iii)$] $F^{-1}(0)=\cup_{j\in\mathbb{N}} A_j$ with $A_j$ closed in $X\times Y$, and the restriction $\pi|_{A_j}$ to $A_j$ of the projection $\pi:X\times Y\to Y$ is a proper map for every $j\in\mathbb{N}$.
 \end{itemize}
 Then, the set \ $\left\{y\in V:0\hbox{ is  a regular value of } F(\cdot,y)\right\}$ \ is a  residual subset of $V$.
\end{theorem}
\smallskip

\begin{proof}[Proof of Theorem~\ref{mainap}]
We show that, for small enough $\rho$, the function $F$ defined in \eqref{eq:F} satisfies the conditions $i)$, $ii)$, and $iii)$ of Theorem \ref{tran} with $\Omega=:U\subset X:=\rn$, $\mathfrak{B}_\rho=:V\subset Y:=\mathfrak{C}^2$, and $Z=\rn$.

$i):$ This assumption is trivially satisfied because every $\cC^1$-map $\o\to\rn$ is Fredholm of index $0$.

$iii):$ Set
$$A_j:=\big(\o_j^-\times\overline{\mathfrak{B}}_{\rho-\frac{1}{j}}\big)\cap F^{-1}(0),\qquad\text{where \ }\o_j^-:=\{x\in\o:\mathrm{dist}(x,\partial\o)\geq \tfrac{1}{j}\}$$
and $\overline{\mathfrak{B}}_r$ is the closed ball of radius $r$ in $\mathfrak{C}^2$ centered at $0$. Clearly, $F^{-1}(0)=\cup_{j\in\mathbb{N}} A_j$. Next, we show that the restriction $\pi|_{A_j}$ to $A_j$ of the projection $\pi:\rn\times \mathfrak{C}^2\to \mathfrak{C}^2$ is a proper map for every $j\in\mathbb{N}$.  Let $\mathfrak{K}$ be a compact subset of $\mathfrak{C}^2$ and $(x_k,\theta_k)$ be a sequence in $A_j\cap\pi^{-1}(\mathfrak{K})$. Then there is a subsequence such that $x_k\to x$ in $\o_j^-$ and $\theta_k\to\theta$ in $\mathfrak{K}$ (because $\mathfrak{K}$ is compact and $(\theta_k)\subset \mathfrak{K}$). This shows that $A_j\cap\pi^{-1}(\mathfrak{K})$ is compact, i.e., $\pi|_{A_j}$ is proper.

$ii):$ Let $(x,\theta)\in \Omega\times \mathfrak{B}_\rho$.  The change of variables $z=y+\theta(y)$ yields
 \begin{align*}
f(x,\theta)= \psi_{\Omega_\theta}(x+\theta(x))=\int_{\mathbb R^n\setminus \Omega_\theta}\frac1{|x+\theta(x)-z|^{2n}}\d z =\int_{\mathbb R^n\setminus \Omega }\frac1{|x-y+\theta(x)-\theta(y)|^{2n}}\left|\mathtt J_\theta(y)\right|\d y,
 \end{align*}
where $|\mathtt J_\theta(y)| = \mathrm{det}(I+D\theta(y))$ is the Jacobian determinant associated to the change of variables. Now, 
\begin{align*} 
F(x,\theta)=\nabla _x f(x,\theta)=-2n\Big(\underbrace{\int_{\mathbb R^n\setminus \Omega }\frac{x-y+\theta(x)-\theta(y)}{|x-y+\theta(x)-\theta(y)|^{2n+2}}\left|\mathtt J_\theta(y)\right|\d y}_{=:\what F(x,\theta)}\Big)(I+D\theta(x)).
\end{align*}
Fix $(x,\theta)\in\Omega\times\mathfrak{B}_\rho$ such that $F(x,\theta)=0$. To prove condition $ii)$, it suffices to show that
\begin{align}\label{claim:ap}
\text{the map $D_\theta F(x,\theta):\mathfrak{C}^2\to\mathbb R^n$ is surjective.  }
\end{align}

Note that, as $I+D\theta(x)$ is an isomorphism for $\rho$ small enough, we have that $\what F(x,\theta)=0$ and, as a consequence,
$$D_\theta F(x,\theta)[\eta]=-2n\big(D_\theta\what F(x,\theta)[\eta](I+D\theta(x)) + \what F(x,\theta)D\eta(x)\big)=-2n D_\theta\what F(x,\theta)[\eta](I+D\theta(x)),\qquad\eta\in\mathfrak{C}^2$$
(to compute this derivative we have used that, for fixed $x$, the map $\mathfrak{C}^2\to\mathscr{L}(\rn,\rn)$ given by $\eta\mapsto D\eta(x)$ is linear and continuous, and the map $\rn\times\mathscr{L}(\rn,\rn)\to\rn$ given by $(z,A)\mapsto zA$ is bilinear, where $\mathscr{L}(\rn,\rn)$ is the space of linear maps from $\rn$ to $\rn$.) Therefore, to show \eqref{claim:ap} it suffices to prove that $D_\theta \what F(x,\theta):\mathfrak{C}^2\to\mathbb R^n$ is surjective.

Using that the derivative of the determinant function $\det:\mathscr{L}(\rn,\rn)\to\r$ at an invertible matrix $A$ is given by $D\det(A)[B]=\det(A)\mathrm{tr}(A^{-1}B)$ we obtain
\begin{align*}
D_\theta \what F(x,\theta)[\eta] &=\int_{\mathbb R^n\setminus \Omega } \frac{ \eta(x)-\eta(y)}{|x-y+\theta(x)-\theta(y)|^{2n+2}}|\mathtt J_{ \theta} y|\d y\\
&\quad+ \int_{\mathbb R^n\setminus \Omega }\frac{x-y+\theta(x)-\theta(y)}{|x-y+\theta(x)-\theta(y)|^{2n+2}}\,\mathrm{tr}\big((I+D \theta(y))^{-1}D\eta(y)\big)\,|\mathtt J_{\theta}(y)|\d y\\ 
&\quad-\int_{\mathbb R^n\setminus \Omega }
 (2n+2)\frac{\langle\eta(x)-\eta(y),x-y+\theta(x)-\theta(y)\rangle\,(x-y+\theta(x)-\theta(y))}{|x-y+\theta(x)-\theta(y)|^{2n+4}} |\mathtt J_{ \theta}y|\d y \\
 &=: I_1[\eta]+I_2[\eta]+I_3[\eta].
\end{align*}
Fix $\zeta\in\partial\o$ and let $\nu_\zeta$ be the exterior unit normal to $\o$ at $\zeta$. Fix a radial function $\chi\in\cC^\infty_c(\rn)$ such that $\chi(0)=1$, $\chi(x)=0$ if $|x|\geq 1$ and $\chi$ is nonincreasing in the radial direction, and set
$$\eta_\zeta(y):=\chi(y)\,\nu_\zeta,\qquad\chi_{\eps,\zeta}(y):=\chi\Big(\frac{y-\zeta}{\eps}\Big)\quad\text{and}\quad\eta_{\eps,\zeta}(y):=\eta_\zeta\Big(\frac{y-\zeta}{\eps}\Big)=\chi_{\eps,\zeta}(y)\,\nu_\zeta\text{ \ \ for \ }\eps>0.$$
Then, for $\eps$ small enough,
\begin{align*}
I_1[\eta_{\eps,\zeta}] &=\int_{B_\eps(\zeta)\setminus\o} \frac{-\chi_{\eps,\zeta}(y)\,\nu_\zeta}{|x-y+\theta(x)-\theta(y)|^{2n+2}}|\mathtt J_{ \theta} y|\d y, \\
I_3[\eta_{\eps,\zeta}] &=-(2n+2)\int_{B_\eps(\zeta)\setminus\o}\chi_{\eps,\zeta}(y)\frac{\langle-\nu_\zeta,\,x-y+\theta(x)-\theta(y)\rangle\,(x-y+\theta(x)-\theta(y))}{|x-y+\theta(x)-\theta(y)|^{2n+4}} |\mathtt J_{ \theta}y|\d y.
\end{align*}
Therefore, 
\begin{equation}\label{eq:I_1}
|I_1[\eta_{\eps,\zeta}]|=O(\eps^n)\qquad\text{and}\qquad|I_3[\eta_{\eps,\zeta}]|=O(\eps^n)\qquad \text{ as }\eps\to 0.
\end{equation}
Now, setting $y=\eps z+\zeta$, \ $\o_\eps:=\{z\in\rn:\eps z+\zeta\in\o\}$, and applying Lebesgue's theorem we derive
\begin{align*}
I_2[\eta_\eps] &=\frac{1}{\eps}\int_{B_\eps(\zeta)\setminus\o}\frac{x-y+\theta(x)-\theta(y)}{|x-y+\theta(x)-\theta(y)|^{2n+2}}\,\mathrm{tr}\Big((I+D\theta(y))^{-1}D\eta_\zeta\Big(\frac{y-\zeta}{\eps}\Big)\Big)\,|\mathtt J_{\theta}(y)|\d y \\
&=\eps^{n-1}\int_{B_1(0)\setminus\o_\eps}\frac{x-\eps z-\zeta+\theta(x)-\theta(\eps z+\zeta)}{|x-\eps z-\zeta+\theta(x)-\theta(\eps z+\zeta)|^{2n+2}}\,\mathrm{tr}\big((I+D\theta(\eps z+\zeta))^{-1}D\eta_\zeta(z)\big)\,|\mathtt J_{\theta}(\eps z+\zeta)|\d z \\
&=\eps^{n-1}\left(\frac{x-\zeta+\theta(x)-\theta(\zeta)}{|x-\zeta+\theta(x)-\theta(\zeta)|^{2n+2}}\,|\mathtt J_{\theta}(\zeta)|\int_{B_1(0)\cap\mathbb{H}_{\nu_\zeta}}\mathrm{tr}\big((I+D\theta(\zeta))^{-1}D\eta_\zeta(z)\big)\d z +o(1)\right),
\end{align*}
where $\mathbb{H}_{\nu_\zeta}:=\{z\in\rn:\langle z,\nu_\zeta\rangle>0\}$ (here we used the smoothness of the domain). Now, $(I+D\theta(\zeta))^{-1}=I+A(\zeta)$, where the entries of $A(\zeta)$ tend to $0$ as $\rho\to 0$. Therefore,
\begin{align*}
c_\zeta:=\int_{B_1(0)\cap\mathbb{H}_{\nu_\zeta}}\mathrm{tr}\big((I+D\theta(\zeta))^{-1}D\eta_\zeta(z)\big)\d z
= \int_{B_1(0)\cap\mathbb{H}_{\nu_\zeta}}\langle\nabla\chi(z),\nu_\zeta\rangle\d z + \int_{B_1(0)\cap\mathbb{H}_{\nu_\zeta}}\mathrm{tr}\big(A(\zeta)D\eta_\zeta(z)\big)\d y.
\end{align*}
The first summand on the right-hand side is strictly negative. So, for $\rho$ sufficiently small (depending only on $\chi$), we have that $c_\zeta<0$. Thus,
$$a_\zeta:=\frac{|\mathtt J_{\theta}(\zeta)|\,c_\zeta}{|x-\zeta+\theta(x)-\theta(\zeta)|^{2n+2}}<0.$$
Summing up, we have shown that
\begin{align*}
D_\theta\what F(x,\theta)[\eta_{\eps,\zeta}]=\eps^{n-1}\Big(a_\zeta (x-\zeta+\theta(x)-\theta(\zeta)) + o(1)\Big)\qquad \text{ as }\eps\to 0.
\end{align*}
All we need to do now is to show that there exist $\zeta_1,\ldots,\zeta_n\in\partial\o$ such that $x-\zeta_1+\theta(x)-\theta(\zeta_1),\ldots,x-\zeta_n+\theta(x)-\theta(\zeta_n)$ are linearly independent. To this end, let $\{e_1,\ldots,e_n\}$ be the canonical basis in $\rn$. Then, there is $t_i>0$ such that $\tilde\zeta_i=x+\theta(x)+t_ie_i\in\partial\o_\theta$ for every $i=1,\ldots,n$. Let  $\zeta_i\in\partial\o$ be such that $\tilde\zeta_i=\zeta_i + \theta(\zeta_i)$. Then, $x-\zeta_i+\theta(x)-\theta(\zeta_i)=-t_ie_i$, so these vectors are linearly independent and, thus, $ii)$ is satisfied. The claim now follows from Theorem \ref{tran}.
\end{proof}

\section{An invertibility result}

Fix $\xi\in\mathbb R^n$ and $\delta>0$. Here we give the details on the invertibility of $\mathscr L_{1,0}:K^\perp_{\delta,\xi}\to K^\perp_{\delta,\xi}$.  We follow \cite[Lemma 2.13]{mp} and here we include a proof for completeness. Recall that
\begin{align*}
 \mathscr L_{\delta,\xi}(\phi)&:=\Pi_{\delta,\xi}^\perp\left\{  \phi-i^* \left[   f'_0 \left(U_{\delta,\xi} \right)\phi \right]\right\}.
\end{align*}

\begin{lemma}\label{B1}
Let $s>\frac{2 n}{n+2}$.  Then $\mathscr L_{1,0}$ is invertible and $\mathscr L_{1,0}^{-1}$ is a continuous function, i.e. there exists a constant $C>0$ such that
\begin{align*}
\left\|\mathscr L_{1,0}\psi\right\| \geq C\|\psi\|,\quad \left\|\mathscr L_{1,0} \psi\right\|_E \geq C\|\psi\|_E \quad \text{ for all }\psi \in K^\perp_{1,0}.
\end{align*}
\end{lemma}
\begin{proof}
\emph{\underline{Surjectivity:}} Let $\psi \in K^\perp_{1,0}$ and let $L_{1,0}:D^{1,2}(\rn)\to D^{1,2}(\rn)$ be given by $L_{1,0} \phi:= \phi-i^*\left[f_0^{\prime}(U_{1,0}) \phi\right]$. Since we know that the kernel of $L$ is $K_{1,0}$, then there is $\phi \in K_{1,0}^{\perp}$ such that $\phi-i^*\left[f_0^{\prime}(U_{1,0}) \phi\right]=L \phi =\psi \in(\operatorname{Ker} L)^{\perp}$. Now we verify that $\phi \in \mathrm{L}^s\left(\mathbb{R}^n\right)$. Indeed, since $\phi \in D^{1,2}(\mathbb{R}^n)$ and $f_0^{\prime}(U_{1,0}) \in \mathrm{L}^t\left(\mathbb{R}^n\right)$ for all $t>n / 4$, we have that $f_0^{\prime}(U_{1,0}) \phi \in \mathrm{L}^q\left(\mathbb{R}^n\right)$ for all $q \in(2 n /(n+6), n / 2)$ and then
\begin{align*}
i^*\left(f_0^{\prime}(U_{1,0})\phi\right) \in \mathrm{L}^s\left(\mathbb{R}^n\right) \quad \text{ for all } s>\frac{2 n}{n+2} .
\end{align*}
This yields the surjectivity.

\emph{\underline{Injectivity:}} If $\mathscr L_{1,0} \phi=0$, then $\phi\in K_{1,0}\cap K^\perp_{1,0}=\{0\}$.

\emph{\underline{Lower bound:}} Since $\mathscr L_{1,0}$ is a continuous map defined on the Banach space $K^\perp_{1,0}$, the open mapping theorem implies that $\mathscr L_{1,0}^{-1}$ is also continuous.
\end{proof}

\begin{lemma}\label{B2}
Let $2^*>s>\frac{2 n}{n+2}$, $d>0$, $\xi\in \Omega$, and $\delta=d\eps^\frac{1}{n}$.  Then $\mathscr L_{\delta,\xi}$ is invertible and $\mathscr L_{\delta,\xi}^{-1}$ is a continuous function uniformly with respect to $\eps$, i.e. there are $C>0$ and $\eps_0$ such that
\begin{align*}
\left\|\mathscr L_{\delta,\xi} \psi\right\|_E \geq C\|\psi\|_E \quad \text{ for all }\psi \in K^\perp_{\delta,\xi}\ \text{and for all $\eps\in(0,\eps_0)$.}
\end{align*}
\end{lemma}
\begin{proof}
The invertibility follows as in Lemma~\ref{B1} and here we only argue the uniform bound on the inverse. Let $\psi\in K^\perp_{\delta,\xi}$ and let $\psi_{\delta,\xi}(x)=\delta^\frac{n-2}{2}\psi(\frac{x-\xi}{\delta}).$ Note that, by \eqref{eq:s_rescaling}, there is $\alpha>0$ such that
\begin{align}\label{beq}
 \left\|\mathscr L_{\delta,\xi} \psi_{\delta,\xi}\right\|
 =\left\|\mathscr L_{1,0} \psi\right\|
 \quad
 \text{ and }
 \quad
 |\mathscr L_{\delta,\xi} \psi_{\delta,\xi}|_s
 =\delta^\alpha|\mathscr L_{1,0} \psi|_s.
\end{align}
Hence, by Lemma~\ref{B1} and \eqref{beq}, we know that $ |\mathscr L_{1,0} \psi|_s
 \geq C(\|\psi\|+|\psi|_s)-\|L_{1,0}\psi\|,$ then
\begin{align*}
 2\left\|\mathscr L_{\delta,\xi} \psi_{\delta,\xi}\right\|_E
 &\geq \left\|\mathscr L_{\delta,\xi} \psi_{\delta,\xi}\right\|+|\mathscr L_{\delta,\xi} \psi_{\delta,\xi}|_s
 =\left\|\mathscr L_{1,0} \psi\right\|+\delta^\alpha|\mathscr L_{1,0} \psi|_s\\
 &\geq\left\|\mathscr L_{1,0} \psi\right\|+C\delta^\alpha(\|\psi\|+|\psi|_s)-\delta^\alpha\|L_{1,0}\psi\|\\
& =(1-\delta^\alpha)\left\|\mathscr L_{1,0} \psi\right\|+C\delta^\alpha(\|\psi\|+|\psi|_s)\\
&\geq C(1-\delta^\alpha)\left\|\psi\right\|+C\delta^\alpha |\psi|_s\geq C(\|\psi\|+\delta^\alpha|\psi|_s)
=C(\|\psi_{\delta,\xi}\|+|\psi_{\delta,\xi}|_s)
\end{align*}
for some $C>0$ and $\delta$ small enough so that $1-\delta^\alpha >\frac{1}{2}$.
\end{proof}

The next Lemma follows as \cite[Lemma B.11]{PST23}, here we present a proof for completeness.  We use the notation introduced in Section~\ref{sch:sec}.

\begin{lemma}\label{Lemma:I_1aux}
For $\gamma\leq 0$ small,
\[
|f_0(U_1-U_2) - f_0(U_1)+f_0(U_2)|_{\frac{(p+1)(1+\gamma)}{p}}=\begin{cases}
O(\delta^{n-2}) & \text{ if } 3\leq n<6,\\
O(\delta^{n-2}) & \text{ if } n=6 \text{ and } \gamma<0,\\
O(\delta^4|\log \delta|^\frac{2}{3}) & \text{ if } n=6 \text{ and } \gamma=0,\\
O(\delta^{\frac{n+2}{2}}) & \text{ if } n>6,
\end{cases}
\]
as $\delta\to 0$. Note that
\begin{align*}
 \frac{(p+1)(1+\gamma)}{p}&=\frac{2n}{n+2}\qquad \text{ for }\gamma=0,\\
  \frac{(p+1)(1+\gamma)}{p}&=\frac{ns}{n+2s}\qquad \text{ for }\gamma=\frac{n (s-2)-2 s}{2 (n+2 s)}<0
\end{align*}
for $\frac{n}{n-2}<s<2^*.$
\end{lemma}
\begin{proof}
Let $r:=\frac{1}{4}\operatorname{dist}(\xi_1,\xi_2)$
We use the notation $\omega_+:=B_{r}(\xi_1)$ and $\omega_-:=B_{r}(\xi_2).$ Let $W=U_1-U_2.$ Observe that
\begin{align}
\int_{\omega_+} | f_0(W) - f_0(U_1)+&f_0(U_2) |^\frac{(p+1)(1+\gamma)}{p} =\int_{\omega_+} \left| |W|^{p-1}W - U_1^p+U_2^p \right|^\frac{(p+1)(1+\gamma)}{p} \nonumber \\
				&\leq C \int_{\omega_+} \left| |W|^{p-1}W - U_1^p \right|^\frac{(p+1)(1+\gamma)}{p} + C\int_{\omega_+} U_2^{(p+1)(1+\gamma)}\label{eq:estimateR_aux2}.
\end{align}
Now, since $|x-\xi_2|\geq c>0$ in $\omega_+$ for some $c>0$,
\begin{align}\label{eq:estimateR_aux3}
\int_{\omega_+} U_2^{(p+1)(1+\gamma)} = \int_{\omega_+} \left( \frac{\alpha_n^{p+1}\delta^n}{(\delta^2+|x-\xi_2|^2)^n}\right)^{1+\gamma} \, dx \leq C |\omega_+| \delta^{n(1+\gamma)} = O(\delta^{n(1+\gamma)}).
\end{align}
Applying the second statement of \cite[Lemma B.1]{PST23} with $q:=p$, $a:=U_1$ and $b:=-U_{\delta,-e_2}$, we obtain
\begin{align}\label{eq:estimateR_aux3.2}
\int_{\omega_+} \left| |W|^{p-1}W-U_1^p \right|^\frac{(p+1)(1+\gamma)}{p}
			&\leq C \int_{\omega_+} \left(|U_1|^{p-1}U_2 + |U_2|^{p}\right)^\frac{(p+1)(1+\gamma)}{p}\\
			& \leq C' \int_{\omega_+} U_1^\frac{(p-1)(p+1)(1+\gamma)}{p} U_2^\frac{(p+1)(1+\gamma)}{p} + C' \int_{\omega_+} |U_2|^{(p+1)(1+\gamma)} \nonumber \\
			&\leq O(\delta^\frac{n(n-2)(1+\gamma)}{n+2}) \int_{\omega_+} |U_1|^\frac{8n(1+\gamma)}{(n+2)(n-2)} + O(\delta^{n(1+\gamma)}). \label{eq:estimateR_aux1}
\end{align}
By \cite[Lemma B.3]{PST23} with $q:=\frac{8n}{(n+2)(n-2)}$ and noting  that $0<q<\frac{2n}{n-2}$ and that $q<\frac{n}{n-2}$ if and only if $n>6$,
\begin{equation}\label{eq:estimateR_aux4}
\int_{\omega_+} U_1^\frac{8n}{(n+2)(n-2)} =\begin{cases}
O(\delta^\frac{n(n-2)}{n+2}) & \text{ if } 3\leq n<6,\\
O(\delta^3|\log \delta|) & \text{ if } n=6,\\
O(\delta^\frac{4n}{n+2}) & \text{ if } n>6.
\end{cases}
\end{equation}
while, for $\gamma<0$ sufficiently small,
\begin{equation}\label{eq:estimateR_aux5}
\int_{\omega_+} U_1^\frac{8n(1+\gamma)}{(n+2)(n-2)} =\begin{cases}
O(\delta^\frac{n(n-2)(1+\gamma)}{n+2}) & \text{ if } 3\leq n \leq 6,\\
O(\delta^\frac{4n(1+\gamma)}{n+2}) & \text{ if } n>6.
\end{cases}
\end{equation}

Going back to \eqref{eq:estimateR_aux1} and combining it with \eqref{eq:estimateR_aux4}--\eqref{eq:estimateR_aux5}, we see that, if $\gamma=0$,
\begin{align*}
\int_{\omega_+} \left| |W|^{p-1}W - U_1^p+U_2^p \right|^\frac{2n}{n+2}
		&=\begin{cases}
O(\delta^\frac{2n(n-2)}{n+2}) & \text{ if } 3\leq n < 6,\\
O(\delta^6|\log \delta|) & \text{ if } n=6,\\
O(\delta^{n(1+\gamma)}) & \text{ if } n>6
\end{cases}
\end{align*}
and, if $\gamma <0$ small,
\begin{align*}
\int_{\omega_+} \left| |W|^{p-1}W - U_1^p+U_2^p \right|^\frac{(p+1)(1+\gamma)}{p}
		&=\begin{cases}
O(\delta^\frac{2n(n-2)(1+\gamma)}{n+2}) & \text{ if } 3\leq n \leq 6,\\
O(\delta^{n(1+\gamma)}) & \text{ if } n>6.
\end{cases}
\end{align*}
Exchanging the roles of $\xi_1$ and $\xi_2$, we have precisely the same type of estimate for
\[
\int_{\omega_-} \left| |W|^{p-1}W - U_1^p+U_2^p \right|^\frac{(p+1)(1+\gamma)}{p}.
\] Finally, since $|x-\xi_1|,|x-\xi_2|\geq c$ for $x\in \Omega \backslash (\omega_+ \cup \omega_-)$, we have
\begin{align*}
\int_{\Omega \backslash (\omega_+ \cup \omega_-)} \left| |W|^{p-1}W - U_1^p+U_2^p \right|^\frac{(p+1)(1+\gamma)}{p} 	&\leq C \int_{\Omega \backslash (\omega_+ \cup \omega_-)} \left( U_1^{(p+1)(1+\gamma)}+ U_2^{(p+1)(1+\gamma)}\right)=O(\delta^{n(1+\gamma)}),
\end{align*}
which concludes the proof.
\end{proof}

\subsection*{Acknowledgments}

We thank the anonymous referees for their careful reading and valuable comments, which helped improve the quality of the paper. M. Clapp and A. Saldaña thank La Sapienza Università di Roma for the kind hospitality during the preparation of this work.


\bigskip

\begin{flushleft}
\textbf{Mónica Clapp}\\
Instituto de Matemáticas\\
Universidad Nacional Autónoma de México \\
Campus Juriquilla\\
76230 Querétaro, Qro., Mexico\\
\texttt{monica.clapp@im.unam.mx}
\medskip

\textbf{Alberto Saldaña}\\
Instituto de Matemáticas,\\
Universidad Nacional Autónoma de México,\\
Circuito Exterior, Ciudad Universitaria,\\
04510 Coyoacán, Ciudad de México, Mexico.\\
\texttt{alberto.saldana@im.unam.mx}
\medskip

\textbf{Angela Pistoia}\\
Dipartimento di Metodi e Modelli Matematici,\\
La Sapienza Università di Roma,\\
Via Antonio Scarpa 16,\\
00161 Roma, Italy.\\
\texttt{angela.pistoia@uniroma1.it}
\end{flushleft}

\end{document}